\documentclass[11pt,reqno]{amsart}
\usepackage{amsmath,graphicx,color}

\usepackage{ifthen}

\usepackage{amssymb,latexsym}
\usepackage{subfigure}
\usepackage{enumitem}
\usepackage{dsfont}
\usepackage{mathtools}
\mathtoolsset{showonlyrefs}

\usepackage{hyperref}
\hypersetup{
urlcolor=black, 
  menucolor=black, 
  citecolor=black, 
  anchorcolor=black, 
  filecolor=black, 
  linkcolor=black, 
  colorlinks=true,
}

\usepackage[numbers,sort&compress]{natbib} 

\numberwithin{equation}{section}

\allowdisplaybreaks

\newcommand{\bfT}{{\mathbf T}}

\newcommand{\bfR}{{\mathbf R}}

\newcommand{\bfD}{\mathbf B} 

\newcommand{\bfM}{{\mathbf M}}

\newcommand{\bfI}{{\mathbf I}}
\newcommand{\bfX}{{\mathbf X}}
\newcommand{\bfY}{{\mathbf Y}}
\newcommand{\bfx}{{\mathbf x}}

\newcommand{\la}{\lambda}

\newcommand{\ds}{distribution}
\newcommand{\beao}{\begin{eqnarray*}}
\newcommand{\eeao}{\end{eqnarray*}}
\newcommand{\beam}{\begin{eqnarray}}
\newcommand{\eeam}{\end{eqnarray}}
\newcommand{\barr}{\begin{array}}
\newcommand{\earr}{\end{array}}

\definecolor{darkblue}{rgb}{.1, 0.1,.8}
\definecolor{darkgreen}{rgb}{0,0.8,0.2}
\definecolor{darkred}{rgb}{.8, .1,.1}
\newcommand{\bco}{\begin{corrolary}}
\newcommand{\eco}{\end{corrolary}}

\newcommand{\E}{\mathbb{E}}
\renewcommand{\P}{\mathbb{P}}
\newcommand{\1}{\mathds{1}}
\newcommand{\R}{\mathbb{R}}
\newcommand{\N}{\mathbb{N}}
\newcommand{\C}{\mathbb{C}}
\newcommand{\bfC}{{\mathbf C}}
\newcommand{\bfS}{{\mathbf S}}
\newcommand{\bfB}{{\mathbf B}}

\newcommand{\Frechet}{Fr\'{e}chet }

\DeclareMathOperator{\e}{e}

\newcommand{\inv}{^{-1}}
\newcommand{\x}{{\mathbf x}}

\newcommand{\X}{{\mathbf X}}
\newcommand{\Y}{{\mathbf Y}}

\newcommand{\dint}{\,\mathrm{d}}

\newcommand{\norm}[1]{\|#1\|}
\newcommand{\twonorm}[1]{\|#1\|}

\newcommand{\vep}{\varepsilon}
\newcommand{\nto}{n \to \infty}

\newcommand{\rhs}{right-hand side}

\newcommand{\tr}{\operatorname{tr}}

\newcommand{\diag}{\operatorname{diag}}

\newcommand{\MP}{Mar\v cenko--Pastur }

\renewcommand{\Im}{\operatorname{Im}}

\newcommand{\T}{\mathbf{T}}
\newcommand{\K}{\mathbf{K}}
\newcommand{\D}{\mathbf{D}}
\newcommand{\lb}{\left(}
\newcommand{\rb}{\right)}
\newcommand{\bfy}{\mathbf{y}} 
\newcommand{\invv}{^{-1}}
\newcommand{\rd}{\mathbf{y}}
\newcommand{\conp}{\stackrel{\mathbb{P}}{\to}}

\newcommand{\sq}{^{\frac{1}{2}}}
\newcommand{\Mj}{\mathbf{M}^{(j)}}
\newcommand{\PR}{\mathbb{P}}
\newcommand{\ty}{\tilde{\mathbf{y}}}

\newtheorem{lemma}{Lemma}[section]

\newtheorem{theorem}[lemma]{Theorem}
\newtheorem{proposition}[lemma]{Proposition}

\newtheorem{corollary}[lemma]{Corollary}
\newtheorem{example}[lemma]{Example}

\newtheorem{remark}[lemma]{Remark}

\newcommand{\cip}{\stackrel{\P}{\longrightarrow}}

\newcommand{\cas}{\stackrel{\rm a.s.}{\rightarrow}}

\newcommand{\eid}{\stackrel{\mathcal{D}}{=}}

\newcommand{\cov}{\operatorname{Cov}}

\newcommand{\Holder}{H\"older}

\begin{document}
\bibliographystyle{acm}
\title[LSD for large sample correlation matrices]
{Limiting spectral distribution for large sample correlation matrices}
\thanks{
This work was partially supported by the German Research Foundation (DFG Research Unit 1735, DE 502/26-2, RTG 2131, \textit{High-dimensional Phenomena in Probability - Fluctuations and Discontinuity} and project number 460867398, DFG Research unit 5381, {\it Mathematical Statistics in the Information Age}).
JH thanks Peter Kevei for important contributions to the proof of Theorem~\ref{cor_i.i.d.}.}
\author[N. Dörnemann]{Nina Dörnemann}
\author[J. Heiny]{Johannes Heiny}
\address{Fakult\"at f\"ur Mathematik,
Ruhr-Universit\"at Bochum,
Universit\"atsstrasse 150,
D-44801 Bochum,
Germany}
\email{nina.doernemann@rub.de, johannes.heiny@rub.de}
\begin{abstract}
In this paper, we consider the empirical spectral distribution of the sample correlation matrix 
and investigate its asymptotic behavior under mild assumptions on the data's distribution, when dimension and sample size increase at the same rate.  
First, we give a characterization for the limiting spectral distribution to follow a \MP~law assuming that the underlying data matrix consists of i.i.d. entries.  
Subsequently, we provide the limiting spectral distribution of the sample correlation matrix when allowing for a dependence structure within the columns of the data matrix. 
In contrast to previous works, the fourth moment of the data may be infinite, resulting in a fundamental structural difference. 
More precisely, the standard argument of approximating the sample correlation matrix by its sample covariance companion breaks down and novel techniques for tackling the challenging dependency structure of the sample correlation matrix are introduced. 
\end{abstract}

\maketitle

\section{Introduction}\setcounter{equation}{0}

Due to a wide variety of applications, measuring and estimating the dependence between two random variables are fundamental problems in
statistics. 
Starting with the early works
of  Pearson \cite{Pearson1920}, Kendall \cite{Kendall1938}, Hoeffding \cite{Hoeffding1948} and Blum \cite{blum1961}, several measures of dependence or association have been introduced and analyzed by numerous authors. An outstanding role is played by Pearson's correlation coefficient, a measure of the linear dependency of two random variables, about which  most students learn early on in their studies. Motivated by its importance for statistical inference and estimation, many works are devoted to its stochastic properties in different frameworks. 
For example, in time series analysis, the notion of correlation plays a vital role in multivariate statistical analysis for  parameter estimation, goodness-of-fit tests, change-point detection, etc.; see for example the classical monographs \cite{brockwell:davis:1991,priestley:1981}.


Consider a $p$-dimensional population $\bfx\in \R^p$ of the form $\bfx= \T^{\frac{1}{2}}\tilde{\bfx}$, where the components of
$\tilde{\bfx}=(\tilde{X}_1,\ldots,\tilde{X}_p)^{\top}$
 are independent random variables that are identically distributed as a centered non-degenerate random variable $\xi$, and $\T^{\frac{1}{2}}$ is the Hermitian square root of a positive semidefinite and non-random matrix $\T\in \R^{p\times p}$.  
For a sample $\mathbf{x}_1,\ldots,\mathbf{x}_n$ from the population, we construct the matrix $\X=\X_n=(\bfx_1,\ldots,\bfx_n)$.
This paper is concerned with the spectral properties of the {\it sample correlation matrix} $\bfR$, which is the empirical version of Pearson correlation for multivariate data. It is given by 
  \begin{equation}\label{eq:defR1}
  \bfR =\bfR_n =\{\diag(\bfS_n)\}^{-1/2}\, \bfS_n\{\diag(\bfS_n)\}^{-1/2}\,,
  \end{equation}
	where $\bfS=\bfS_n= n^{-1} \X\X^{\top}$ denotes the {\it sample covariance matrix}.

With the rapid advancements of data collection devices, many modern fields such as biological engineering, telecommunications and finance require the analysis of high-dimensional data sets where the dimension and the sample size are of comparable magnitude. Be aware that traditional multivariate analysis, as outlined in the textbooks of \cite{anderson2003, muirhead1982}, relies on the assumption that the dimension remains fixed and thus is negligible compared to the sample size. For this reason, results from traditional multivariate analysis are typically not applicable in other regimes. Spurred by these problems, new analysis tools for high-dimensional data were developed in recent years. We add to this line of literature and concentrate on the regime where the dimension-to-sample-size ratio $p/n$ tends to a positive constant as $n,p\to \infty$.
For a detailed discussion of typical applications where such an assumption is natural, we refer to \cite{johnstone:2001,donoho:2000, Fan2006, Johnstone2006}.

\subsection{Related literature}

When considering such high-dimensional data sets and associated random matrices $\X$, the main focus of interest has been on the asymptotic
properties of the eigenvalues of the sample covariance matrix $\bfS$. These have been well analyzed in random matrix theory since the pioneering work \cite{marchenko:pastur:1967} where it is shown that for independent and identically distributed (i.i.d.)\ components of $\x$ with finite variance and as $p/n\to \gamma\in (0,\infty)$ the empirical spectral distribution of $\bfS$ converges weakly to the
  celebrated \MP law $\sigma_{MP,\gamma}$ with parameter $\gamma$.
If $\gamma \in (0,1]$,  $\sigma_{MP,\gamma}$  has density,
\begin{equation*}
  f_\gamma(x) =
  \frac{\sqrt{( b_\gamma-x)(x-a_\gamma )}}{2\pi \gamma x} \1_{
      [a_\gamma, b_{\gamma}]}(x)\,,\quad  x\in\R, 
\end{equation*}
with $a_\gamma=(1-\sqrt{\gamma})^2 $  and $b_\gamma=(1+\sqrt{\gamma})^2$. If $\gamma>1$, the \MP law has an additional point mass $1-1/\gamma$ at $0$.
For non-i.i.d.\ components of $\x$, that is $\T \neq \bfI_p$ (the $p$-dimensional identity matrix), the limiting spectral distribution (LSD) can be characterized in terms of an integral equation for its Stieltjes transform.  
Subsequently, several  ground-breaking results such as  the
  convergence of the largest eigenvalue $\lambda_1(\bf S)$ and the smallest eigenvalue $\lambda_p(\bfS)$ to the edges of  the \MP law \citep{BaiYin88a,tikhomirov:2015},
  asymptotic  normality of linear spectral statistics of $\bfS$ \citep{BS04},  or its edge
  universality towards the Tracy-Widom law \citep{johnstone:2001,Peche2012,PillaiYin2014} were established. 
	Apart from the convergence of $\lambda_p(\bfS)$ all those results require a finite fourth moment of $\xi$.

In case of infinite fourth moments, the theory for the eigenvalues and eigenvectors of $\bfS$ is quite different from the aforementioned \MP~theory. For example, if the distribution of $\xi$ is regularly varying with index $\alpha\in (0,4)$, the properly normalized largest eigenvalue of $\bfS$ converges to a \Frechet distribution with parameter $\alpha/2$. A detailed account on the developments in the heavy-tailed case can be found in \cite{davis:heiny:mikosch:xie:2016,heiny:mikosch:2017:iid,basrak:heiny:jung:2020,auffinger:arous:peche:2009}.
	
	For the sample correlation matrix $\bfR=\{\diag(\bfS)\}^{-1/2} \bfS \{\diag(\bfS)\}^{-1/2}$, the situation gets more complicated because of the specific nonlinear dependence structure caused by the normalization $\{\diag(\bfS)\}^{-1/2}$,  which makes the analysis of this random matrix quite challenging. As a consequence, the study of the high-dimensional sample correlation matrix is more recent and somewhat limited. 
In case $\T=\bfI$ and $\xi$ has zero mean, variance equal to one and finite fourth moment, Jiang \cite{jiang:2004} (see also \cite{elkaroui:2009, heiny:mikosch:2017:corr}) showed that the \MP~law $\sigma_{MP,\gamma}$ is still valid for the sample correlation matrix $\bfR$. Moreover, the first result for the linear spectral statistics of $\bfR$ was proved in \cite{Gao2017}. Under $\T\neq \bfI_p$ and $\E[\xi^4]<\infty$, spectral properties of $\bfR$ were derived in \cite{elkaroui:2009, heiny:2022}.
A central step in the proofs of all these results is to approximate the sample correlation $\bfR$ with $\{\diag(\T)\}^{-1/2} \bfS \{\diag(\T)\}^{-1/2}$. Indeed, assuming $p/n\to \gamma\in (0,\infty)$ and uniform boundedness of $\norm{\T}$ (the spectral norm of $\T$), it is known that $\E[\xi^4]<\infty$ implies 
\begin{equation}\label{eq:maindiagappr}
\norm{\diag(\bfS)-\diag(\T)} \cas 0\,, \qquad n \to \infty\,,
\end{equation}
with equivalence in the case $\T=\bfI_p$ (see \cite[Theorem 1.2]{heiny:2022} and \cite[Lemma 2]{bai:yin:1993}). Therefore, under finite fourth moment the normalization $\{\diag(\bfS)\}^{-1/2}$ in \eqref{eq:defR1} can be replaced with $\{\diag(\T)\}^{-1/2}$ and consequently $\norm{\bfR-\{\diag(\T)\}^{-1/2} \bfS \{\diag(\T)\}^{-1/2}}$ converges to zero almost surely as $\nto$. Due to the self-normalization property of the sample correlation matrix, one may without loss of generality assume that $\diag(\T)=\bfI_p$. This means that the first order spectral properties of $\bfR$ and $\bfS$ are the same if $\E[\xi^4]<\infty$. In particular, the smallest and largest eigenvalues of $\bfR$ and $\bfS$ have the same limits and the LSDs of $\bfR$ and $\bfS$ coincide.

To the best of our knowledge, high-dimensional sample correlation matrices under infinite fourth moment $\E[\xi^4]=\infty$ have only been considered in the i.i.d.\ case $\T=\bfI_p$. On the one hand, if the distribution of $\xi$ is in the domain of attraction of the normal distribution, \cite{bai:zhou:2008} proved that the LSD of $\bfR$ is $\sigma_{MP,\gamma}$. On the other hand, if the distribution of $\xi$ is in the domain of attraction of an (infinite variance) $\alpha$-stable distribution with $\alpha\in (0,2)$, then the LSD is the $\alpha$-heavy \MP~law \cite{heiny:yao:2020}.

\subsection*{Our contributions} 
The contributions of this paper are twofold. 
\begin{itemize}
\item In the i.i.d.\ case $\T=\bfI_p$, we provide a characterization for the limiting spectral distribution of $\bfR$ to follow the \MP~law $\sigma_{MP,\gamma}$. More precisely, we show that the latter is equivalent to the convergence of certain quadratic forms, which are analyzed in detail.
\item For a larger class of population correlation matrices $\T$ and  assuming $\E[|\xi|^{2+\delta}]<\infty$ for some $\delta>0$, we prove that the empirical spectral distributions of $\bfR$ converge weakly almost surely to a generalized \MP~law.  In contrast to previous works, the $4$th moment of the data may be infinite, resulting in a fundamental structural difference. 
More precisely, the standard argument of approximating the sample correlation matrix $\bfR$ by its sample covariance companion $\bfS$ breaks down and novel techniques for tackling the challenging dependency structure of $\bfR$ are introduced.
\end{itemize}

\subsubsection*{Structure of this paper}
This work is organized as follows. In the remainder of this section, we introduce some necessary notation and our model. In Section \ref{sec_i.i.d.}, we present our results when assuming that the underlying data matrix consists of i.i.d.\ entries, while Section \ref{sec_dependent} is devoted to the dependent case. Section \ref{sec_proof_i.i.d.} consists of two further Subsections \ref{sec:proofMP} and \ref{sec_aux_i.i.d.}, where the first one contains the main proofs for the results provided in Section \ref{sec_i.i.d.} and in the latter, one can find some auxiliary results. The proof of Theorem~\ref{thm}, which is the main result of Section \ref{sec_dependent},  consists of several steps as outlined in Section \ref{sec_strat_proof} and is therefore deferred to Section \ref{sec_proof_dep}. Finally, the Appendix provides some rather technical details that are needed for the proof of Theorem~\ref{thm}. 


\subsubsection*{Notation}
For any matrix $\bfC$, the spectral (or operator) norm $\twonorm{\bfC}$ is the square root of the largest eigenvalue of $\bfC\bfC^\star$, where $\bfC^\star$ is the complex conjugate of $\bfC$. Moreover, if $\bfC$ is a square matrix, $\diag(\bfC)$ denotes the diagonal matrix which has the same diagonal as $\bfC$. $\bfI_p$ is the $p$-dimensional identity matrix and if the dimension is clear from the context, we will sometimes just write $\bfI$. 

For any $p\times p$ matrix $\bfC$ with real eigenvalues, we denote its ordered eigenvalues by 
$\la_1(\bfC) \ge \cdots \ge \la_p(\bfC)$.
Hence, we have $\twonorm{\bfC}=\sqrt{\la_1(\bfC \bfC^\star)}$. 
Writing $\1$ for the indicator function, the {\em empirical spectral distribution} of $\bfC$ is defined by 
\begin{equation*}
F^{\bfC}(x)= \frac{1}{p}\; \sum_{i=1}^p \1{\{ \lambda_i(\bfC)\le x \}}, \qquad x\in  \R\,.
\end{equation*}
Let $\mu$ be a finite measure on the real line. Its {\em Stieltjes transform} $s_\mu$ is given by
\begin{align*}
		s_\mu(z) = \int \frac{1}{x-z} \mu ( d x),\qquad  z\in\mathbb{C}^+\,,
	\end{align*}
where $\C^+$ are the complex numbers with positive imaginary part. If $\mu = F^{\mathbf{C}}$ denotes an empirical spectral distribution of some matrix $\mathbf{C}$, we abbreviate $s_{F^{\mathbf{C}}} = s_{\mathbf{C}}$.  Moreover, we write $s_n$ for the Stieltjes transform of $F^{\bfR}$.

We also make use of the notation $a \lesssim b$ for real numbers $a,b\in\R$ if there exists some constant $c >0$ independent of $n\in\N$ with the property $a \leq c b$. Note that while the constant $c$ is not allowed to vary with $n\in\N$, it may depend on $z\in\mathbb{C}^+$.

\subsection{The model}\label{sec:model}

Consider a $p$-dimensional population $\mathbf{\tilde{x}}=(\tilde{X}_1,\ldots,\tilde{X}_p)^{\top}\in\R^p$,
  where the coordinates  $\tilde{X}_i$  are independent
  random variables and  identically distributed as a centered  random variable $\xi$ satisfying  $\E [| \xi |^{1+\delta}] < \infty$ for some $\delta >0$ and $\E[\xi^2]=1$ whenever $\E[\xi^2]<\infty$. 
  For a sample $\tilde{\mathbf{x}}_1,\ldots,\tilde{\mathbf{x}}_n$ from the population we construct the 
 matrix $\tilde{\mathbf{X}}=\tilde{\mathbf{X}}_n=(\tilde{\mathbf{x}}_1,\ldots,\tilde{\mathbf{x}}_n)=(\tilde{X}_{ij})_{1\le    i\le p; 1\le j  \le n}$ and set 
$$\X=\T^{\frac{1}{2}} \tilde{\X}=(X_{ij})_{1\le    i\le p; 1\le j\le n}\,,$$
where the so--called population correlation matrix $\T = \T_n \in\R^{p\times p}$ denotes a symmetric positive semidefinite non-random matrix satisfying $\diag(\T)=\bfI_p$ and $\T^{\frac{1}{2}} = (U_{kl})_{1 \leq k,l \leq p}$ its Hermitian square root. 
The sample covariance matrix $\bfS$
  and the sample correlation matrix $\bfR$ are then given as follows:
  \begin{equation*}
	\begin{split}
  \bfS &= \bfS_n =\frac1n \sum_{i=1}^n \T^{\frac{1}{2}} \tilde{\mathbf{x}}_i \tilde{\mathbf{x}}_i^\top \T^{\frac{1}{2}}=\frac1n \X\X^{\top}\,,
  \\
  \bfR &=\bfR_n =\{\diag(\bfS_n)\}^{-1/2}\, \bfS_n\{\diag(\bfS_n)\}^{-1/2}= \Y \Y^\top\,.
	\end{split}
  \end{equation*}
   Here the self-normalized  matrix $\Y=\Y_n=(Y_{ij})_{1\le    i\le p; 1\le j  \le n}  $ for the correlation matrix has
  entries 
  \begin{align} \label{def_yij}
 Y_{kj}=Y_{kj}^{(n)} 
   = \frac{X_{kj} }{\sqrt{ X_{k1}^2 + \ldots + X_{kn}^2 }}
  = & \frac{\sum\limits_{l=1}^p U_{kl} \tilde{X}_{lj}}{\sqrt{ X_{k1}^2 + \ldots + X_{kn}^2 }}\,.
  \end{align}
By construction, the rows of $\bfY$, which we denote by $\ty_1, \ldots, \ty_p \in\R^{1\times n}$, possess Euclidean norm equal to one. They are independent if and only $\T=\bfI_p$. At first sight, the structure of $\bfS$ and $\bfR$ looks similar as they can both be written as some matrix ($n^{-1/2}\X$ resp. $\bfY$) times its transpose. Note, however, that the columns of $\bfX$ are i.i.d.\ whereas the columns of $\bfY$, which we denote by $\bfy_1, \ldots, \bfy_n \in \R^p$,  are not independent due to the joint normalization term $\{\diag(\bfS_n)\}^{-1/2}$.

Throughout this paper, we assume the asymptotic regime where the sample size $n$ and the dimension $p$ tend to infinity simultaneously, i.e.,
\begin{equation*}
p=p_n \to \infty \quad \text{ and } \quad \frac{p}{n}\to \gamma\in (0,\infty)\,,\quad \text{ as } \nto\,. 
\end{equation*}	
We usually suppress the dependence on $n$ in our notation and write $\bfR,\bfS,\X,\Y, \T$ for the matrices $\bfR_n,\bfS_n,\X_n,\Y_n,\T_n$, respectively.

\section{The i.i.d. case}\setcounter{equation}{0} \label{sec_i.i.d.}

Throughout this section, we consider the model introduced in subsection \ref{sec:model} with $\T$ being the $p\times p$ identity matrix $\mathbf{I}_p$. In our analysis of the LSD of the sample correlation matrices $\bfR=\Y\Y^{\top}$,
an important role will be played by the resolvent 
\begin{equation}\label{eq:D}
\bfD(z)
 := (\bfY^\top\bfY - \ty_1^\top \ty_1 -z \bfI_n)^{-1}\,, \quad z\in\mathbb{C}^+,
\end{equation}
where $\ty_1=(Y_{11},\ldots, Y_{1n})$ is the first row of $\bfY$.
We are ready to state our main result for the sample correlation matrix in the i.i.d.\ case
; compare \cite[Theorem~2.1]{yaskov:2016} for a corresponding statement about the sample covariance matrix.

\begin{theorem}\label{thm:MP}
Assume $p/n\to \gamma$, as $n\to\infty$.
Then the following two statements are equivalent:
\begin{itemize}
\item[(i)] The empirical spectral distributions $F^{\bfR}$ converge weakly almost surely to the \MP~law with parameter $\gamma$.
\item[(ii)] For all $z\in\mathbb{C}^+$ one has
\begin{equation}\label{eq:condii}
W_n(z):=\ty_1 \bfD(z) \ty_1^\top - \frac{1}{n} \tr\big( \bfD(z) \big) \cip 0\,, \qquad \nto\,.
\end{equation}
\end{itemize}
\end{theorem}

Theorem \ref{thm:MP} shows that the LSD of the sample correlation matrix depends on the behavior of the random variables $W_n(z)$, which are a quadratic forms in the self-normalized random vector $\ty_1$ and the matrix $\bfD(z)$. Since $\ty_1,\ldots, \ty_p$ are i.i.d., $\bfD(z)$ and $\ty_1$ are independent.  The next lemma collects some basic properties of the sequence $(W_n(z))$.

\begin{lemma}\label{lem:propw}
Assume $p/n\to \gamma$, as $n\to\infty$, and let $z\in\mathbb{C}^+$. Then the random variables $W_n(z), n\ge 1,$ satisfy $|W_n(z)|\le 2/\Im(z)$ and $\lim_{\nto} \E[W_n(z)] =0$. 
Moreover, for the decomposition 
$W_n(z) = W_{n,1}(z) + W_{n,2}(z)$ with
\begin{equation*}
W_{n,1}(z):= \ty_1 \diag(\bfD(z)) \ty_1^\top - \frac{1}{n} \tr\big( \bfD(z) \big) \quad \text{ and } \quad W_{n,2}(z):=W_n(z)-W_{n,1}(z)\,
\end{equation*}
it holds, as $\nto$,
\begin{equation}\label{eq:EW2}
\E[|W_{n,1}(z)|^2]=  n\E[Y_{11}^4] \bigg( \frac{1}{n} \E\Big[\sum_{i=1}^n |\big(\bfD(z)\big)_{ii}|^2\Big]  -\frac{1}{n^2} \E\Big[\big|\tr \big(\bfD(z) \big) \big|^2 \Big] \bigg) +o(1),
\end{equation}
and $\E [ | W_{n,2} (z) |^2 ] =o(1)$.
\end{lemma}

\begin{remark}{\em Some comments about the decomposition $W_n(z) = W_{n,1}(z) + W_{n,2}(z)$ are in place.
\begin{enumerate}
\item Since $Y_{11}^2+\ldots+Y_{1n}^2=1$ by definition of $\bfR$, we have $\E[Y_{11}^2]=1/n$ and as a consequence $\E[W_{n,1}(z)]=0$. From equation \eqref{eq:EWn2} later on we see that $|\E[W_{n,2}(z)]|\lesssim n\E[Y_{11}Y_{12}]$. It is interesting to note that $\E[Y_{11}Y_{12}]\ge 0$ with equality if and only if the distribution of $\xi$ is symmetric, that is $\xi\eid -\xi$ (see \eqref{eq:EY1Y2} for details).
Therefore, $W_n(z)$ is centered if and only if the distribution of $\xi$ is symmetric. \\
\item An application of Markov's inequality and the last part of Lemma~\ref{lem:propw} yield for $\vep>0$,
$$\P(|W_{n,2}(z)|>\vep)\le \vep^{-2} \E [ | W_{n,2} (z) |^2 ]\to 0 \,, \qquad \nto\,.$$
Hence, statement {\it(ii)} in Theorem~\ref{thm:MP} is equivalent to: 
\begin{itemize}
\item[{\it(ii')}] For all $z\in\mathbb{C}^+$ one has $W_{n,1}(z) \cip 0$, as $\nto$.
\end{itemize}
\end{enumerate}
}\end{remark}
Now we provide two equivalent sufficient conditions for the convergence of the empirical spectral distributions of the sample correlation matrix to the \MP law.

\begin{theorem}\label{cor_i.i.d.}
Assume $p/n\to \gamma$, as $n\to\infty$. Then the empirical spectral distributions $F^{\bfR}$ converge weakly almost surely to the \MP~law with parameter $\gamma$ if
\begin{equation}\label{eq:esfe}
\lim_{\nto} n \E[Y_{11}^4] =0\,.
\end{equation}
Furthermore, condition \eqref{eq:esfe} is equivalent to $\xi$ being in the domain of attraction of the normal distribution. 
\end{theorem}

\begin{proof}
Let $z\in\mathbb{C}^+$. Our strategy is to show $W_n(z)\cip 0$, which by Theorem \ref{thm:MP} establishes the convergence of the empirical spectral distributions. By Markov's inequality, the condition $W_n(z)\cip 0$ is implied by $\E[|W_{n,1}(z)|^2]=o(1)$ and $\E[ |W_{n,2}(z)|^2] =o(1)$, where the latter follows from Lemma \ref{lem:propw}. We note that 
\begin{equation}\label{eq:bound32}
\max_{i=1,\ldots,n} \big|\big(\bfD(z)\big)_{ii}\big| \le \norm{\bfD(z)}\le \frac{1}{\Im(z)}\,,
\end{equation}
where \eqref{eq:le1} was used for the last inequality.
A combination of \eqref{eq:EW2} and \eqref{eq:bound32} yields that $\E[|W_{n,1}(z)|^2]= O(n\,\E[Y_{11}^4])+o(1)$ which tends to zero as $\nto$ if \eqref{eq:esfe} holds. 

Next, we turn to the second part of the theorem. 
By \cite[Theorem 5.4]{fuchs:joffe:teugels:2001} (with $X = \xi^2$)
the convergence \eqref{eq:esfe}
is equivalent to the relative stability of $\xi^2$, that is 
\[
 \int_0^x \P (\xi^2 > y ) \mathrm{d} y 
\]
is a slowly varying function. We say that a function $L$ is slowly varying (at infinity) if $L(tx)/L(x)\to 1$, as $x\to \infty$, for all $t>0$. Now we have
\[
 \int_0^x \P ( \xi^2 > y ) \mathrm{d} y = 2 \int_0^{\sqrt{x}} \P(|\xi|> u) u\, \mathrm{d} u
 =: h_2 ( \sqrt{x} ).
\]
If $h_2(\sqrt{x})$ is slowly varying then $h_2(x)$ is slowly varying, which
by Theorem 1.1 in \cite{kevei:2021} is equivalent to the slow variation of
\[
 V_2(x) = \int_{[0,x]} y^2 \mathrm{d} \P(|Y|\le y).
\]
The latter is the characterization of the domain of attraction of the normal law
(Theorem 8.3.1 in \cite{bingham:goldie:teugels:1987}).
\end{proof}

\subsection{The role of $W_n(z)$}
In this subsection, we will investigate the influence of the random variables $W_n(z), n\ge 1$, on the LSD in a more general situation. We have the following result.  
\begin{theorem}\label{thm:2ndchar}
Assume $p/n\to \gamma$, as $n\to\infty$. Then the Stieltjes transform $s_n$ of $F^{\bfR}$ satisfies
\begin{equation*}
-z \E[s_n(z)]= \E\left[ \frac{1}{1+W_n(z) +\gamma \E[s_n(z)]-z^{-1} (1-\gamma)} \right] +o(1)
, \qquad z\in \C^+, \nto\,.
\end{equation*}
\end{theorem}
From Theorem~\ref{thm:2ndchar} we immediately get the next corollary.
\begin{corollary}\label{cor:ened}
Assume $p/n\to \gamma$ and $\E[s_n(z)]\to s(z)$, as $n\to\infty$, where $s(z)$ is the Stieltjes transform of some probability measure. Then $s(z)$ satisfies the equation
\begin{equation}\label{eq:charW1}
-z\, s(z)= \lim_{\nto} \E\left[ \frac{1}{1+ W_n(z) +\gamma s(z)-z^{-1} (1-\gamma)} \right]\,
, \qquad z\in \C^+\,.
\end{equation}
and the limit on the \rhs~exists.
\end{corollary}
If we replace $W_n(z)$ with $0$, then \eqref{eq:charW1} is the usual equation for the Stieltjes transform of the \MP~law; see, e.g., \cite{bai:silverstein:2010}. In general, it is not possible to replace $W_n(z)$ with its expectation, unless of course $W_n(z)\cip 0$.
As seen in Theorem \ref{cor_i.i.d.},  the latter is implied by $n\E[Y_{11}^4]\to 0$. In this case, Theorem~\ref{thm:MP} confirms that the LSD of the sample correlation matrices is the \MP~law.

We proceed by investigating $n\E[Y_{11}^4]$ more closely.
For all $y, \beta >0$, we have
		\begin{align}\label{formula_inv}
			\frac{1}{y^\beta} = \frac{1}{\Gamma(\beta)} \int\limits_0^\infty 
			\exp(-ty) t^{\beta -1} \dint t\,, 
		\end{align}
where $\Gamma$ denotes the Gamma function. Combining this representation with Fubini's theorem, we deduce
\begin{equation}\label{eq:formulagine}
\E[ Y_{11}^4 ] =\E\left[ \frac{X_{11}^4}{\big(X_{11}^2+\cdots+X_{1n}^2\big)^2} \right] = \int_0^{\infty} 
 t (\E[\e^{-t \xi^2}])^{n-1}   \E[\xi^{4}\e^{-t \xi^2} ] \, \dint t\,.
\end{equation}
The integrand involves the Laplace transform $t \mapsto \E[\e^{-t \xi^2}]$ of $\xi^2$ and its second derivative; see \cite{fuchs:joffe:teugels:2001} for further details. By definition of $Y_{11}$, the value of $n\,\E[ Y_{11}^4 ]$ lies in the interval $(0,1)$. In this context, the limiting case \eqref{eq:esfe} can be seen as an extreme scenario. It turns out that all limiting values in the above range $(0,1)$ are possible. 
Proposition~1 in \cite{mason:zinn:2005} asserts that the \ds\ of $\xi^2$ is in the 
domain of attraction of an $\alpha/2$-stable distribution with parameter $0< \alpha <2$ if and only if
\begin{equation}\label{eq:limita<2}
\lim_{\nto} n\,\E[Y_{11}^4] =1-\frac{\alpha}{2}\,.
\end{equation}
Examples of such distributions include the Pareto distribution with parameter $\alpha$ and Student's $t$-distribution with $\alpha$ degrees of freedom.
In this case we may obtain a limiting spectral distribution which additionally depends on the value $\alpha$. 

\begin{example}{\em
Let the \ds\ of $\xi^2$ be in the 
domain of attraction of an $\alpha/2$-stable distribution with parameter $0< \alpha <2$ and assume $\xi \eid -\xi$. In this setting, the authors of \cite{heiny:yao:2020} proved that the empirical spectral distributions $F^{\bfR_n}$
  converge weakly in probability to  some probability law $H_{\alpha,\gamma}$, which they termed $\alpha$-heavy MP law with
  parameter $\gamma$.
$H_{\alpha,\gamma}$ is entirely
determined by its moment sequence
$\mu_{k}(\alpha,\gamma)=\int x^k\dint H_{\alpha,\gamma}(x)$, $k\ge 1$.
The exact expression for  $\mu_k(\alpha,\gamma)$
is rather involved (see \cite{heiny:yao:2020}).  
Since $H_{\alpha,\gamma}$ uniquely characterizes its Stieltjes transform $s^{\alpha,\gamma}$,
we obtain by Corollary~\ref{cor:ened} that $s^{\alpha,\gamma}(z)$ satisfies equation \eqref{eq:charW1}. 
}\end{example}

\section{Adding dependency}\setcounter{equation}{0} \label{sec_dependent}
In this section, we study the limiting spectral distribution of the sample correlation matrix $\bfR$ when allowing for a more sophisticated population correlation matrix $\T$.
	 For this purpose, let $\mathcal{I}(i) =\mathcal{I}^{(n)}(i)= \{ 1 \leq k \leq p : U_{ik} = U_{ki} \neq 0\}$ denote the set of indices of non-vanishing entries in the $i$th row or column of the Hermitian square root $\mathbf{U}=\T^{\frac{1}{2}} = (U_{kl})_{1 \leq k,l \leq p}$ ($1 \leq i \leq p$). Among other assumptions stated below, we will impose a sparsity condition on $\mathbf{U}$ in terms of controlling the cardinality of the set $\mathcal{I}(i)$.
	 We propose the following conditions for deriving the limiting distribution of $F^{\bfR}$, which are discussed in Remark~\ref{rem_ass}. 
	\begin{enumerate}[label=(A\arabic*) ]
		\item  $\inf_{n\in \N} \lambda_{\min} (\T_n) > 0 $ 
		. \label{a_eigen_T}
		\item $F^{\T_n} \to H$ almost surely, as $n\to \infty$, where $H$ is a non-random c.d.f..
			\item \label{a_sparse_U} 
		$ \sup\limits_{n\in\N} \max\limits_{1 \leq i \leq p} | \mathcal{I}(i) | <\infty $.
			\item \label{a_2mom} The random variables $\tilde{X}_{ij}$ are i.i.d. according to $\xi$, which satisfies $\E[\xi]=0, \E[\xi^2]=1$ and $\E | \xi |^{2+\delta} < \infty$ for some $\delta>0$. 
		\item \label{a_high_dim} $ p=p_n \to \infty \quad \text{ and } \quad p/n\to \gamma\in (0,\infty)\,,\quad \text{ as } \nto\,. $
	\end{enumerate}

\noindent The following is our main result in the dependent case.
	\begin{theorem} \label{thm}
	Under assumptions \ref{a_eigen_T}-\ref{a_high_dim}, 
	 the empirical spectral distributions $F^{\bfR}$ converge, as $n\to\infty$, weakly almost surely to the generalized \MP~law $F^{\gamma, H}$ with parameter $(\gamma,H)$, whose Stieltjes transform $s=s(z)$ 
		is the unique solution to the equation
		\begin{align} \label{eq_s}
			s(z) = \int \frac{1}{\lambda ( 1 - \gamma - \gamma z s(z) ) - z} dH(\lambda)\,, \qquad z\in\mathbb{C}^+.
		\end{align}
	\end{theorem}
A strategy of the proof of Theorem \ref{thm} outlining  our novel technical tools is discussed in Section~\ref{sec_strat_proof}.  The complete proof of Theorem \ref{thm} can be found in Section \ref{sec_proof_dep}.
	\begin{remark} \label{rem_ass}{\em 
	\begin{enumerate}
	\item We have the following implications for moments of entries of $\mathbf{Y}$. 
	To begin with, note that $\E [Y_{kj}^2] = \frac{1}{n}$, since 
$Y_{k1}^2+\cdots+Y_{kn}^2 = 1 $
 and $Y_{k1}, \ldots, Y_{kn}$ are identically distributed for each $k\in\{1, \ldots, p\}$.
	Additionally, the first moment satisfies
	\begin{align} \label{first_mom}
		\lim\limits_{n\to\infty}  \max\limits_{1 \leq k \leq p} n | \E  [ Y_{k1}  ] |  = 0
	\end{align}
	and for the fourth moment, we have
	\begin{align} \label{4th_mom}
		\lim\limits_{n\to\infty}  \max\limits_{1 \leq k \leq p} n  \E  [ Y_{k1}^4  ]   = 0
	\end{align}
	(see Proposition \ref{thm_first_mom} and Proposition \ref{thm_4mom} given later).
	For further details about the moments of the self-normalized random variables $Y_{kj}$, we refer the reader to Section \ref{sec_mom_y}.  
	\item Instead of imposing the existence of the $(2+\delta)$th moment of the generic element $\xi$, it is seen from the proof of Theorem \ref{thm} (more precisely, from the proof of Lemma \ref{lem_quad_form}) that \eqref{first_mom} and \eqref{4th_mom} are sufficient conditions. Hence, we could replace assumption \ref{a_2mom} by the following assumption:
	\begin{enumerate}
	\item[\textnormal{(A4')}] 
	The random variables $\tilde{X}_{ij}$ are i.i.d. according to $\xi$, which satisfies $\E[\xi]=0, \E[\xi^2]=1$, and \eqref{first_mom} and \eqref{4th_mom} hold true. 
	\end{enumerate}
	Consequently, the assertion of Theorem \ref{thm} holds true under assumptions \ref{a_eigen_T}-\ref{a_sparse_U}, (A4'), \ref{a_high_dim}. This observation draws a noteworthy connection to the results in the i.i.d. framework presented in Section \ref{sec_i.i.d.}, where we observed that the asymptotic behavior of $n \E [ Y_{11}^4]$ plays a crucial role for the limiting spectral distribution of $\bfR$ to follow a \MP law. More precisely, we recover the first part of Theorem \ref{cor_i.i.d.} (under the additional assumption $\E [ \xi^2 ] =1$) by applying Theorem \ref{thm} with $\T = \mathbf{I}$. 
	\item The sparsity condition imposed on the square root $\mathbf{U}$ in \ref{a_sparse_U} implies an analogue condition for the population correlation matrix $\T$, since it is seen that
	\begin{align*}
		T_{ik} = \sum_{l=1}^p U_{lk} U_{li} \neq 0 \,, \qquad 1 \leq i \neq k \leq p,
	\end{align*}
	implies $ \mathcal{I}(i) \cap \mathcal{I}(k) \neq \emptyset $. As a consequence, we obtain that the spectral norm of $\T$ 
	is bounded uniformly in $n\in\N$, that is, 
		\begin{align*}
		\sup\limits_{n\in\N} \| \T_n  \| < \infty .
	\end{align*}
	\end{enumerate}
	}\end{remark}

	\begin{example}{\em 
	
		  	 \begin{figure}[!ht]
    \centering
             \includegraphics[width=0.49\columnwidth, height=0.35\textheight]{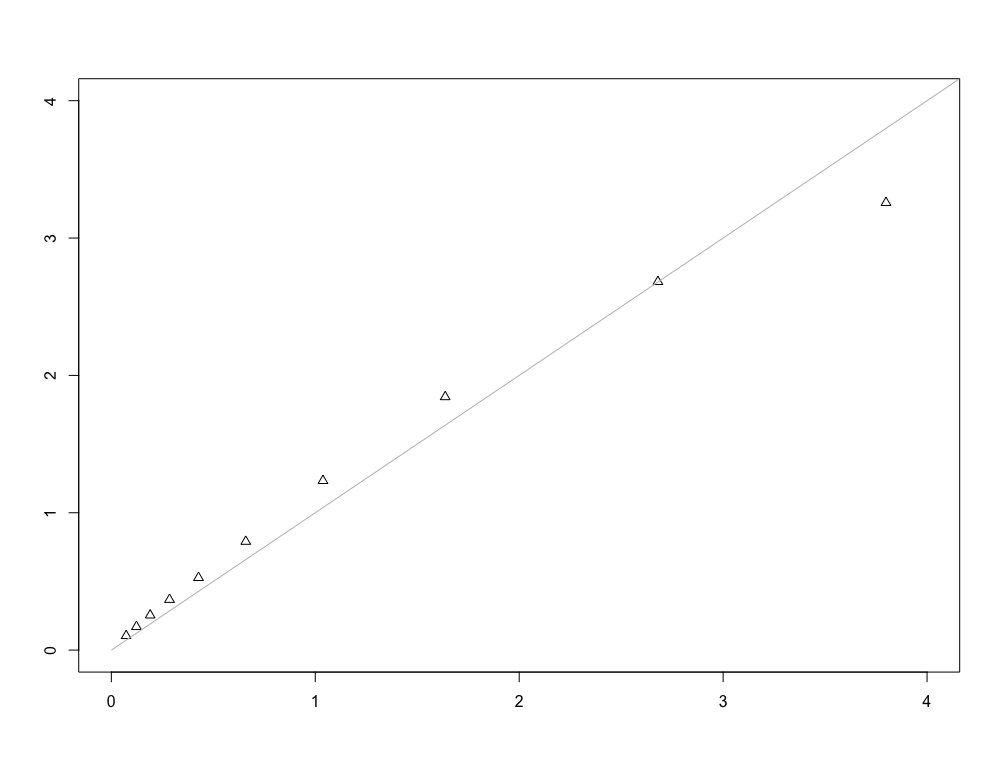}
             \includegraphics[width=0.49\columnwidth, height=0.35\textheight]{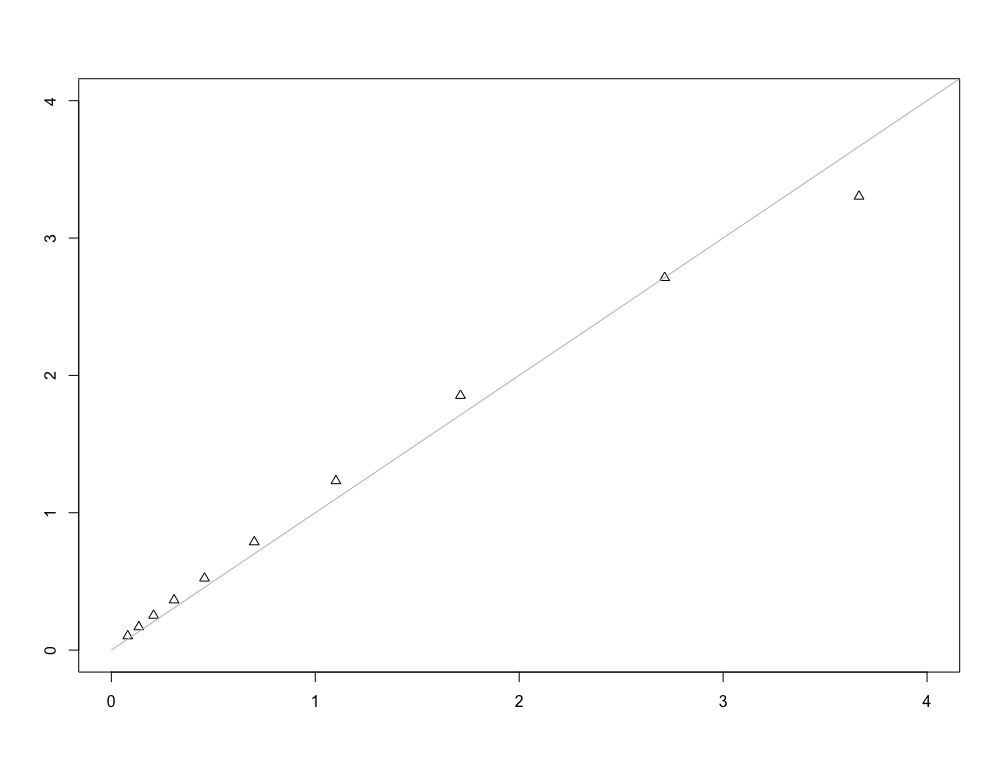}    
              \includegraphics[width=0.49\columnwidth, height=0.35\textheight]{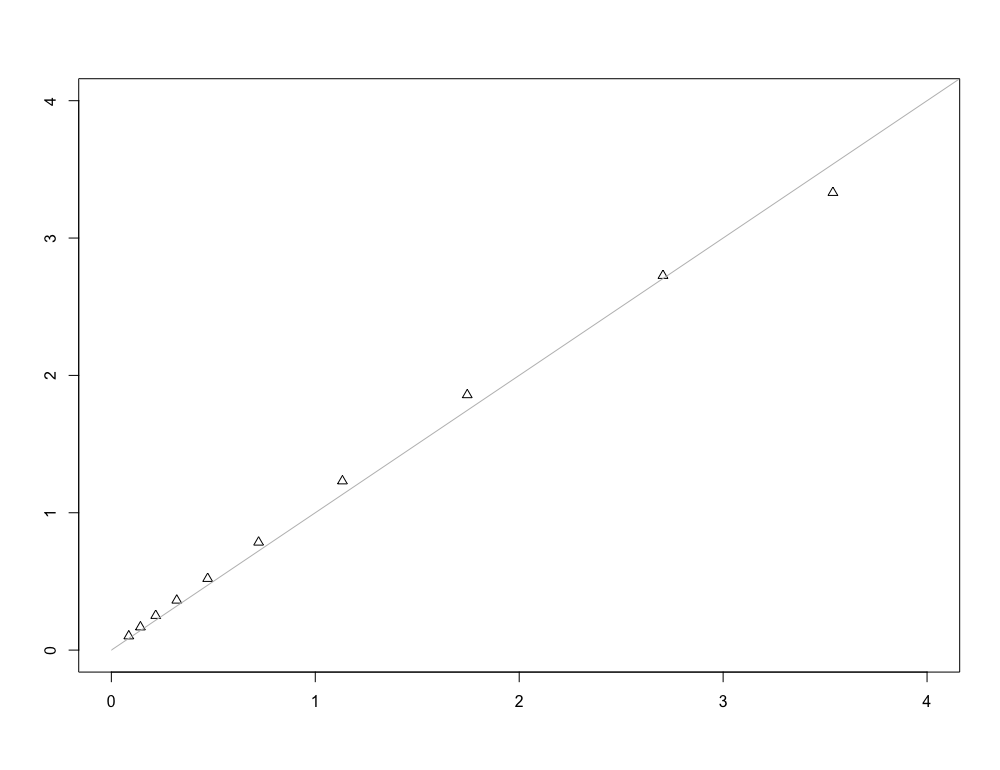} 
             \includegraphics[width=0.49\columnwidth, height=0.35\textheight]{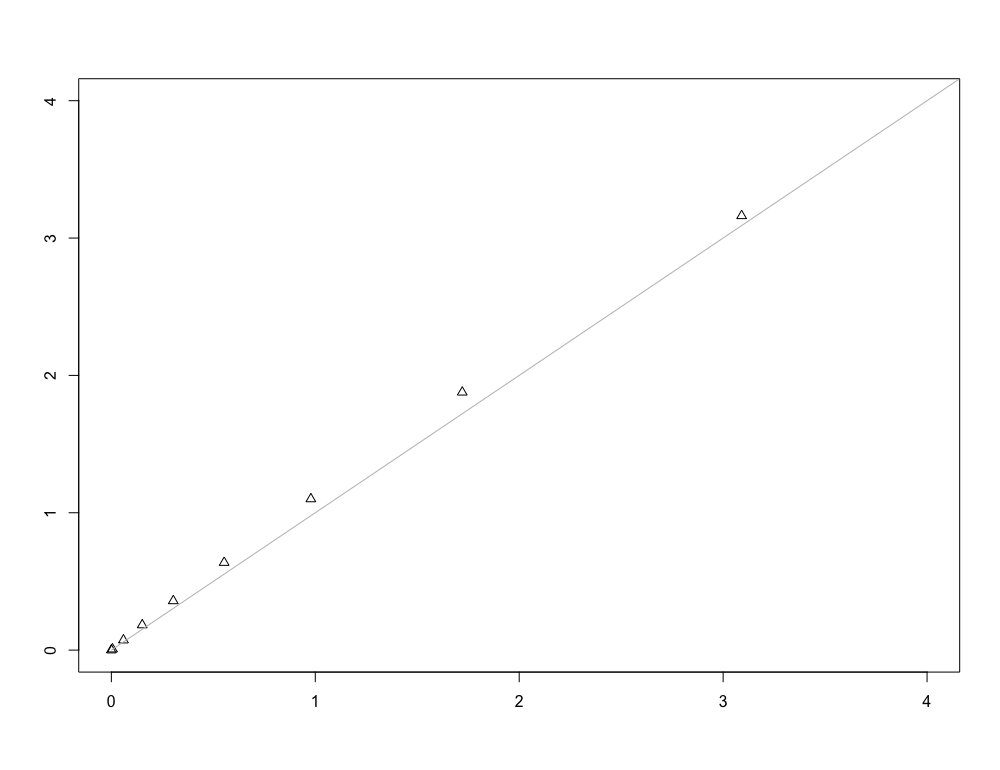}

         \caption{   Simulated $q$-quantiles of $F^{\mathbf{R}_n}$ for normally distributed data ($x$-axis) and standardized $t$-distributed data with degree of freedom 3 ($y$-axis) for $q \in\{0.1, 0.2, \ldots, 0.9, 0.95\}$, $n=200, ~p=100$ (top left), $n=400, ~p=200$ (top right), $n=800, ~p=400$ (bottom left), $n=400, ~p=500$ (bottom right). }
    \label{fig1}
    \end{figure}

	 We continue this section with a small simulation study illustrating the finite-sample behavior of the empirical spectral distribution $F^{\mathbf{R}_n}$ using different distributions for the generic element $\xi$. In Figure \ref{fig1}, we display the empirical quantiles of $F^{\mathbf{R}_n}$ for standard normal distributed data and for standardized $t$-distributed data with $3$ degrees of freedom, that is, the fourth moment of the latter distribution does not exist. By Theorem \ref{thm}, we know that, in both cases, $F^{\mathbf{R}_n}$ admits the same limiting spectral distribution.  Since for general $\T$, the limiting spectral distribution of $\mathbf{R}_n$ given in Theorem \ref{thm} has no closed form, we decided to choose the normal case as a reference.
	In particular, we study the case of two non-vanishing subdiagonals for the population correlation matrix $\mathbf{T}= (T_{ij})_{1\leq i,j \leq p}$, which has entries
	\begin{align*}
		T_{ij} = \begin{cases} 
		1\,, & \text{if } i=j, \\ 
		0.5\,, & \text{if } i = j -1 \textnormal{ or } i - 1 = j, \\
		0.25\,, & \text{if } i = j -2 \textnormal{ or } i - 2 = j,
		\end{cases} 
		~ \qquad 1 \leq i,j \leq p,
	\end{align*}
	for various values of $n$ and $p$ and simulated the $q$-quantiles for each of the two empirical spectral distributions based on $300$ simulation runs, where $q \in\{0.1, 0.2, \ldots, 0.9, 0.95\}$. \\
	 One can observe that in the case of the $t$-distribution, the distribution $F^{\mathbf{R}_n}$ admits heavier tails in comparison to the normal case indicated by the outlying point corresponding to the $ 95 \%$ -quantile, not surprisingly due to much less regularity. However, especially for large sample size and dimension, the empirical quantiles behave very similar, which reflects the asymptotic result provided in Theorem \ref{thm}. 
	}\end{example} 
	
\subsection{Strategy of the proof} \label{sec_strat_proof}
	In the following, we point out the main ideas for proving the convergence of the Stieltjes transform $s_n(z)$ of $\bfR$ to $s(z)$ for all $z\in \C^+$. Note that when only assuming a finite moment of order $(2+\delta)$ for $\xi$, the standard argument \eqref{eq:maindiagappr} of approximating $\bfR=\Y\Y^{\top}$ via $\bfS=n^{-1}\X\X^{\top}$ may break down. Additionally, the matrix $\Y$ admits a challenging dependence structure both among the rows and the columns. The latter highlights a fundamental difference to the covariance case (see, e.g., \cite{bai:zhou:2008}), where the columns of $\mathbf{X}$ are independent, or to the correlation matrix $\bfR$ for $\T = \mathbf{I}$ studied in Section \ref{sec_i.i.d.}, where the rows of $\Y$ are independent. Consequently, our setting demands for a more sophisticated analysis. \\ 
	At first sight, the proof of the convergence of the random part $s_n(z) - \E [ s_n(z)]$ (Lemma \ref{lem_step1}) makes use of standard tools such as martingale decomposition, Burkholder's and Azuma's inequality. However, note that additional subtle difficulties arise, since $\D_j(z)=\bfR - \mathbf{y}_j \mathbf{y}_j^\top-z\bfI$ does depend on the $j$th column $\bfy_j$ ($1 \leq j \leq n$) and consequently, $(\E_j - \E_{j-1}) \D_j\inv(z)$ does not vanish in general. (Here, $\E_j$ denotes the conditional expectation with respect to $\bfx_1, \ldots, \bfx_j$ for $1 \leq j \leq n$ and $\E_0$ denotes the usual mean.)
	Thus, instead of $\D_j\inv(z)$, a more suitable approximation of the resolvent $\D\inv(z)=(\bfR-z\bfI)^{-1}$ independent of $\bfy_j$ is needed whose properties are analyzed in this part of the proof.   \\
	Moreover, when considering the non-random part $\E [ s_n(z) ] - s(z)$ (Lemma \ref{lem_step2}), the crucial part lies in considering quadratic forms of the type
	\begin{align} \label{quad_form}
		\bfy_1^\top \mathbf{A} \D_1\inv(z) \bfy_1 - n\inv \tr \T \mathbf{A} \D_1\inv(z)\,
	\end{align}
for some matrix $\mathbf{A}\in\mathbb{C}^{p\times p}$ independent of $\bfy_1$,
	which turns out to be a delicate task since, due to the self-normalization, the components of $\bfy_1$ are not independent and the matrix $\D_1\inv(z)$ depends also on $\bfy_1$. Via an approximation argument (Lemma \ref{lem_Dhat}), we achieve that  $\D_1\inv(z)$ can be replaced by a further matrix independent of $\bfy_1.$ 
	In order to control \eqref{quad_form}, the sparsity assumption \ref{a_sparse_U} turns out to be essential for the proof of Lemma \ref{lem_quad_form}. 
	Then, the asymptotic behavior of \eqref{quad_form} is determined by the (mixed) moments of the self-normalized random variables $Y_{k1}$, $1 \leq k \leq p$, which are analyzed using an integral representation trick for $Y_{k1}^m$, $m\in\N$. 
	This approach connects the fourth moment of $Y_{k1}$ with its square's Laplace transform which enables us to determine its asymptotic order. For the first moment, a similar technique is applied. 
	We emphasize that this analysis calls for particular attention since the random variables $Y_{k1} = Y_{k1}^{(n)}$ form a triangular array ($1 \leq k \leq p =p_n, ~ n\in\N$) and we need uniform bounds over $k,n\in\N.$ For details on the moments of $Y_{k1}$, we refer the reader to Section \ref{sec_mom_y} and Appendix \ref{appendix_laplace}.
	
\subsection{Outlook}
We conclude this section with some comments on potential applications.
To the best of our knowledge, this paper is the first one that establishes some limiting properties of a large sample correlation
  matrix from a population with infinite fourth moment in the dependent case.  As explained in Section \ref{sec_strat_proof},
  the study of such a correlation matrix is involved and most tools available in the literature on sample covariance matrices are not anymore applicable here.   Theoretical tools
  on this topic are not numerous indeed. Much needs to be done for the
  development of such techniques to facilitate the proof of results which have high impact on statistical
  applications. For example, much of statistical inference based on sample
  correlation matrices requires
  a central 
  limiting theorem for linear statistics of eigenvalues,
  which seems  beyond reach in the infinite fourth moment case at the moment. We remark that in this case, a central limit theorem for linear spectral statistics is also not available for the sample covariance matrix. In fact, the properly normalized trace of $\bfS$, for example, converges to an infinite variance stable distribution whenever $\xi^2$ is in the domain of attraction of an $\alpha$-stable law with $\alpha\in (0,2)$.
  
  Furthermore, it should be expected that linear spectral statistics for $\bfR$ (if at all valid in the infinite fourth moment framework) will depend on more characteristics than just the self-normalized fourth moment $\E[Y_{11}^4]$. For example, in the i.i.d.\ case the results in \cite{heiny:parolya:2022} indicate that for symmetric $\xi$ satisfying $\P(|\xi|>x)=x^{-\alpha} L(x)$ for some slowly varying function $L$ a phase transition appears in the asymptotic behavior of the logarithm of the determinant of $\bfR$ at $\alpha=3$ which is the border of finite and infinite third moment $\E[|\xi|^3]$. 
 On the other hand, if $\E[\xi^4]<\infty$, the central limit theorem for the logarithm of the determinant of $\bfR$ depends on the distribution of $\xi$ only through $\E[\xi^4]$ even in the dependent case \cite{heiny:parolya:kurowicka:2021}.
 
  In view of the evolution of the literature on sample covariance matrices, it is clear that 
  any further
  development on asymptotic properties of the sample correlation
  matrix from a population with infinite fourth moments will require the
  knowledge of the LSD as developed in this paper. In this sense, this work can be viewed as a
  meaningful first step toward thoughtful  statistical applications that might  be developed subsequently.	
	
	\section{Proofs of results in the i.i.d.\ case} \label{sec_proof_i.i.d.} 
	Throughout this section, we work under the assumptions of Section \ref{sec_i.i.d.}. 
	In particular, recall that $\T = \mathbf{I} \in \R^{p\times p}$, $\E[|\xi|^{1+\delta}]<\infty$ for some $\delta>0$, and $\ty_1\in \R^{1\times n}$ denotes the first row of $\Y$. For convenience of notation, let $\Y_{-1}\in \R^{(p-1)\times n}$ be the matrix $\Y$ with the first row removed. 
	
\subsection{Proofs of Theorem \ref{thm:MP}, Lemma~\ref{lem:propw} and Theorem~\ref{thm:2ndchar}}\label{sec:proofMP}

\begin{proof}[Proof of Theorem~\ref{thm:MP}]
Let $z\in\mathbb{C}^+$ and recall that $s_n$ denotes the Stieltjes transform of $F^{\bfR}$. We will show that both (i) and (ii) are equivalent to 
\begin{equation}\label{eq:3.3}
\E\left[ \frac{1}{1+\ty_1 \bfD(z) \ty_1^\top} \right] =
\E\left[ \frac{1}{1+\gamma s_n(z)-z^{-1} (1-\gamma)} \right] + o(1)\,, \quad \nto\,.
\end{equation} 

First, we prove the equivalence of \eqref{eq:3.3} and (i). Assuming \eqref{eq:3.3} we have by Lemma \ref{lem:expectedtransform1} that 
\begin{equation}\label{eq:3.4}
\E\left[ \frac{1}{1+\ty_1 \bfD(z) \ty_1^\top} \right]
= \frac{1}{1+\gamma \E[s_n(z)]-z^{-1} (1-\gamma)} + o(1)\,.
\end{equation} 
 Considering Lemma \ref{lemma3.2}, we obtain
 \begin{equation}\label{eq:3.5}
-z \E[s_n(z)] = \frac{1}{1+\gamma \E[s_n(z)]-z^{-1} (1-\gamma)} + o(1)\,.
\end{equation} 
Therefore, $\E[s_n(z)]$ converges to the unique positive solution of the equation
\begin{equation}\label{eq:smp}
-zS=\frac{1}{1+\gamma S - z^{-1} (1-\gamma)}\,,
\end{equation}
which is well known to be the Stieltjes transform of the \MP~law. See for instance \cite[p.~13, eq.~(2.9)]{yao:zheng:bai:2015} or \cite[p.~55]{bai:silverstein:2010} where the equivalent formulation $\gamma z S^2 + (z+\gamma -1) S +1 = 0$ is preferred. By Lemma \ref{lem:stieltjes}, statement (i) follows.

Now, let us assume (i). Lemmas \ref{lem:stieltjes} and \ref{lem:expectedtransform1} and the fact that the Stieltjes transform of the  \MP~law satisfies \eqref{eq:smp} imply 
 \eqref{eq:3.5}.
 Thanks to Lemmas \ref{lemma3.2} and \ref{lem:stieltjes} we get  \eqref{eq:3.3}.

Next, we prove the equivalence of \eqref{eq:3.3} and (ii). Using the fact that
$\bfR=\Y\Y^\top$ and $\Y^\top\Y$ have the same non-zero eigenvalues we obtain
\begin{equation}\label{eq:sgdsglop}
s_n(z)
= \frac{n}{p} s_{\Y^\top\Y}(z)+ \frac{1}{z} \Big( \frac{n}{p}-1 \Big)\,.
\end{equation}
By Lemma 6.9 in \cite{bai:silverstein:2010}, we have
\begin{align*}
	\left| \tr  \bfD (z)  - 
	\tr \lb \bfY^\top\bfY -z \bfI_n)^{-1}\rb \right|
	\leq \frac{1}{\Im(z)}.
\end{align*}
Since 
$p/n \to \gamma$, this implies 
\begin{equation}\label{eq:resa}
\gamma s_n(z)-\frac{1-\gamma}{z} - \frac{1}{n} \tr\bfD(z) \cas 0\,, \qquad \nto\,
\end{equation}
in view of \eqref{eq:sgdsglop}. Therefore, equation \eqref{eq:3.3} is equivalent to 
\begin{equation}\label{eq:3.3'}
\E\left[ \frac{1}{1+\ty_1 \bfD(z) \ty_1^\top} \right] =
\E\left[ \frac{1}{1+n^{-1} \tr\big( \bfD(z) \big)} \right] + o(1)\,, \quad \nto\,.
\end{equation}

Assume (ii). 
Using \eqref{eq:le4} and \eqref{eq:le5}, we have
\begin{align*}
 \left| \frac{1}{1+n^{-1} \tr\big( \bfD(z) \big)}-\frac{1}{1+\ty_1 \bfD(z) \ty_1^\top} \right| 
= &  \left| \frac{W_n(z)}{[1+n^{-1} \tr\big( \bfD(z) \big)] [1+\ty_1 \bfD(z) \ty_1^\top]} \right| 
\\ 
\le & \lb \frac{|z|}{\Im(z)} \rb^2 |W_n(z)|\,.
\end{align*}
Hence, boundedness of $|W_n(z)|$ and $W_n(z)\cip 0$ yield \eqref{eq:3.3'}.

Finally, assume \eqref{eq:3.3'} holds. From Lemma \ref{le:eyy} we have
\begin{equation}\label{eq:awe}
\E \big[\ty_1 \bfD(z) \ty_1^\top |\Y_{-1} \big]- \frac{1}{n} \tr \big( \bfD(z) \big)  \cas  0\,, \quad \nto\,.
\end{equation}
Combined with \eqref{eq:3.3'} this means that 
\begin{equation*}
\lim_{\nto} \E\left[ \frac{1}{1+\ty_1 \bfD(z) \ty_1^\top} \right] -
\E\left[ \frac{1}{1+\E \big[\ty_1 \bfD(z) \ty_1^\top |\Y_{-1} \big]} \right] = 0\,.
\end{equation*}
Then Lemma \ref{lemma3.4} yields 
\begin{equation*}
\ty_1 \bfD(z) \ty_1^\top -\E \big[\ty_1 \bfD(z) \ty_1^\top |\Y_{-1} \big]\cip 0\,, \qquad \nto\,,
\end{equation*}
where we used that the random variable $\ty_1 \bfD (z) \ty^\top_1$ is bounded (for details, see \eqref{eq:Wbounded}). 
In conjunction with \eqref{eq:awe}, we obtain (ii).
\end{proof}

\begin{proof}[Proof of Lemma~\ref{lem:propw}]
Let $z\in\mathbb{C}^+$. First, we note that $\E[Y_{11} Y_{12}]=o(n^{-1})$ holds due to Lemma \ref{lem_mixed_mom_i.i.d.}. Consequently, an application of Lemma \ref{le:eyy} yields $\E[W_n(z)] \to 0$.
Let us check that $W_n(z)$ is a bounded random variable. We have 
\begin{equation}\label{eq:Wbounded}
\begin{split}
|W_n(z)| &\le |\ty_1 \bfD(z) \ty_1^\top| + \frac{1}{n} |\tr\big( \bfD(z) \big)| 
 \leq 2 \| \bfD (z) \| \leq \frac{2}{\Im(z)} ,
\end{split}
\end{equation}
where we used $\ty_1 \ty_1^\top =1$, \eqref{eq:le1} and the estimate
\begin{align}
    \label{est_quad_form}
    | \bfx^\top \mathbf{A} \bfx | \leq \bfx^\top \bfx \| \mathbf{A} \|, \qquad\bfx\in\R^n, ~\mathbf{A} \in \C^{n\times n}.
\end{align}

For simplicity of notation, we write $\bfD=(d_{ij}) = \bfD(z)$ and $Y_i$ instead of $Y_{1i}$.
We decompose $W_n(z) = W_{n,1}(z) + W_{n,2}(z),$
where
\begin{align*}
    W_{n,1}(z) = & \sum\limits_{i=1}^n d_{ii} \lb Y_i^2 - n\inv \rb \quad \text{ and } \quad
    W_{n,2}(z) =  \sum\limits_{\substack{i,j=1, \\ i\neq j}}^n d_{ij} Y_i Y_j.
  \end{align*}
  We first investigate the second absolute moment of $W_{n,1}(z)$. This gives
\begin{equation*}
\begin{split}
\E[| W_{n,1}(z)| ^2|\bfD]&= \E\left[ \Big|\sum_{i=1}^n d_{ii} \big(Y_i^2 - n\inv \big) \Big|^2 |\bfD \right] \\
&= \sum_{i=1}^n |d_{ii}|^2 (\E[Y_i^4]-\tfrac{1}{n^2}) + \sum_{\substack{i,j=1, \\ i\neq j}}^n d_{ii} \overline{d_{jj}} (\E[Y_i^2 Y_j^2]-\tfrac{1}{n^2}).
\end{split}
\end{equation*}
Observe that (by an application of \eqref{eq:le1})
\begin{equation}\label{eq:dii}
\max_{i,j=1,\ldots,n} |d_{ij}| \le \twonorm{\bfD} \le \frac{1}{\Im(z)},
\end{equation}
and additionally, by using \eqref{est_quad_form}, 
\begin{equation}\label{eq:dij}
\Big| \sum_{i,j=1}^n d_{ij} \Big|= \left|  1_n^\top \bfD 1_n \right| \le \frac{n}{\Im(z)}\,.
\end{equation}
Taking expectation of the identity $(Y_1^2+\cdots +Y_n^2)^2=1$, we obtain
\begin{equation}\label{eq:gtesaw}
\E[Y_{1}^2 Y_{2}^2] -n^{-2} =  \frac{1}{n^2 (n-1)} - \frac{\E[Y_1^4]}{n-1} . 
\end{equation}
By \eqref{eq:dii}-\eqref{eq:gtesaw},
it follows that 
\begin{equation*}
\E[|W_{n,1}(z)|^2|\bfD]= \E[Y_1^4] \sum_{i=1}^n |d_{ii}|^2  - \frac{\E[Y_1^4]}{n-1} \sum_{\substack{ i,j=1, \\ i\neq j}}^n d_{ii} \overline{d_{jj}}
+o(1)\,.
\end{equation*}
Taking expectation, we get \eqref{eq:EW2}.

For the second absolute moment of $W_{n,2}(z)$ we obtain
\begin{align}\label{eq:dgsde}
\E[|W_{n,2}(z)|^2|\bfD]&=\sum\limits_{\substack{i,j,k,\ell=1, \\ i\neq j, k\neq \ell}}^n d_{ij} \overline{d_{k\ell}} \,\E[Y_i Y_jY_kY_{\ell}] \notag\\
&\lesssim n^2 \E[Y_1 Y_2Y_3Y_4] + n^2 |\E[Y_1^2 Y_2Y_3]| + n \E[Y_1^2 Y_2^2] \,, 
\end{align}
where we used \eqref{eq:dij} as well as
\begin{align*}
\sum_{i,j=1}^n |d_{ij}|^2 &= \tr\big(\bfD \bfD^{\star}\big)\le n \norm{\bfD}^2\le \frac{n}{\Im(z)^2}\, \quad \text{ and }\\
\sum_{i,j=1}^n |d_{ij}|&\le n \max_{i=1,\ldots,n} \sum_{j=1}^n |d_{ij}|\le \frac{n^{3/2}}{\Im(z)}\,.
\end{align*}
For the last inequality the equivalence of the row-sum and operator matrix norms and \eqref{eq:dii} were utilized.
It remains to show that each term on the \rhs\ of \eqref{eq:dgsde} tends to zero. From $(Y_1^2+\cdots +Y_n^2)^2=1$ we immediately deduce that $n^2\E[Y_1^2 Y_2^2]\lesssim 1$. Next, we observe that
$\E[Y_1Y_2]=\E[Y_1Y_2(Y_1^2+\cdots +Y_n^2)]=(n-2) \E[Y_1^2Y_2Y_3] + 2 \E[Y_1^3Y_2]$ which  in combination with \Holder's inequality and Lemma~\ref{lem_mixed_mom_i.i.d.} yields
\begin{align*}
n^2 |\E[Y_1^2 Y_2Y_3]|&\lesssim
n\E[Y_1Y_2] +n|\E[Y_1^3Y_2]| \le n\E[Y_1Y_2]+n \sqrt{\E[Y_1^4] \E[Y_1^2Y_2^2]} =o(1)\,.
\end{align*}
Finally, we apply \cite[Lemma~3.1]{gine:goetze:mason:1997} to bound $\E[Y_1Y_2Y_3Y_4]$ and get
\begin{align*}
0\le n^2 \E[Y_1 Y_2Y_3Y_4] &\lesssim n^2 \sqrt{ \E[Y_1Y_2]} \sqrt{ \E[Y_1Y_2Y_3Y_4Y_5Y_6]}\\ 
&\le n^2 \sqrt{ \E[Y_1Y_2]} \sqrt{ \E[Y_1^2Y_2^2Y_3^2]}\le  \sqrt{ n\,\E[Y_1Y_2]}\,.
\end{align*}
Since $n\E[Y_1Y_2]$ tends to zero by Lemma \ref{lem_mixed_mom_i.i.d.}, the proof is complete. 
\end{proof}
\begin{remark}\label{remark_w2}{\em
For the off-diagonal part $W_{n,2}(z)$, we can additionally show that
\begin{align*}
    \E\big[W_{n,2}(z) \,\big|\bfD(z)\big]= \E [ Y_1 Y_2] 
  \sum\limits_{\substack{i,j=1, \\ i\neq j}}^n d_{ij},
\end{align*}
where we note that $\E [ Y_1 Y_2]\ge 0$ by \eqref{eq:EY1Y2}.
Hence, using \eqref{eq:dij} 
 it follows
\begin{align}\label{eq:EWn2}
    \big|\E\big[W_{n,2}(z)\big] \big|\le  \E [ Y_1 Y_2]\, \frac{n}{\Im(z)}\,.
\end{align}}
\end{remark}

\begin{proof}[Proof of Theorem~\ref{thm:2ndchar}]
From Lemma \ref{lemma3.2} we have
\begin{equation*}
-z \E[s_n(z)]= \E\left[ \frac{1}{1+W_n(z)+\tfrac{1}{n} \tr(\bfD(z))} \right]\,, \qquad z\in \C^+.
\end{equation*} 
In view of \eqref{eq:resa}, this is equivalent to 
\begin{equation*}
-z \E[s_n(z)]= \E\left[ \frac{1}{1+W_n(z) +\gamma s_n(z)-z^{-1} (1-\gamma)} \right] +o(1)
, \qquad z\in \C^+.
\end{equation*}
An application of Lemma \ref{lem:expectedtransform1} concludes the proof.
\end{proof}

\subsection{Auxiliary results} \label{sec_aux_i.i.d.}

Recall the definition of $Y_{ij}$ in \eqref{def_yij} and that $\E[|\xi|^{1+\delta}]<\infty$ for some $\delta>0$. The following result states that the mixed moment of $Y_{11}$ and $Y_{12}$ decreases at a sufficiently fast rate under the assumptions proposed for the i.i.d. case.
\begin{lemma} \label{lem_mixed_mom_i.i.d.}
For any $\delta' < \delta$ it holds 
\begin{align*}
    \lim\limits_{n\to\infty} n^{1+\delta'} \E [ Y_{11} Y_{12} ] =0. 
\end{align*}
In particular, we have 
$   \lim\limits_{n\to\infty} n \E [ Y_{11} Y_{12}] = 0$.
\end{lemma} 
\begin{remark}{\em
 If $\xi$ is in the domain of attraction of the normal law, then $\E[Y_{11} Y_{12}]=o(n^{-2})$ (see \cite{gine:goetze:mason:1997}).  It is interesting to remark that $\E[Y_{11} Y_{12}]=0$ if and only if the distribution of $\xi$ is symmetric. Also note that, in general, an application of \Holder's inequality is not quite enough for the weaker statement in the last line of Lemma \ref{lem_mixed_mom_i.i.d.} since $\E[Y_{11} Y_{12}]\le \E[Y_{11}^2]=n^{-1}$. For the proof of Lemma~\ref{lem_mixed_mom_i.i.d.}, we use an integral representation involving the Laplace transform of $\xi^2$. 
}\end{remark}

\begin{proof}[Proof of Lemma \ref{lem_mixed_mom_i.i.d.}]
	Using \eqref{formula_inv} and Fubini's theorem, we obtain the representation 
	\begin{align}\label{eq:EY1Y2}
	   n^{1+\delta'} \E [ Y_{11} Y_{12} ] = n^{1+\delta'} \int\limits_0^\infty 
	    \lb \E \left[ \xi \exp(-s\xi^2) \right] \rb^2 \varphi^{n-2} (s) ds ,
	\end{align}
	where $\varphi(s) = \E [ \exp(-s \xi^2) ], s>0$,
	denotes the Laplace transform of $\xi^2$. 
	Let $\varepsilon >0. $ We observe that
	\begin{align*}
	    	\int\limits_{\varepsilon}^{\infty} \lb \E \left[ \xi \exp(-s\xi^2)  \right] \rb^2 \varphi^{n-2} (s) ds 
	    	\lesssim 	\int\limits_{\varepsilon}^{\infty} s^{- \frac{1}{2}}  \E \left[ | \xi | \exp(-s\xi^2) \right] \varphi^{n-2} (s) ds,
	\end{align*}
where we used that $\E | \xi | \exp ( - s \xi^2)  \lesssim s^{-1/2}$ by maximizing over $\xi$.
	Thus, we recover a simplified version of the integrand in \eqref{int_first_mom} in the proof of Proposition \ref{thm_first_mom} in this case. 
    As a consequence of Lemma 3.1 in \cite{fuchs:joffe:teugels:2001}, we have
    \begin{align*}
        n^{1+\delta'} \varphi^{n-2} (s) 
        \leq \frac{ n^{1+\delta'} \varphi^n(\varepsilon)  }{\varphi^2 (\varepsilon) } =o(1), ~\quad  \varepsilon < s < \infty.
    \end{align*}
    We are allowed to apply the dominated convergence theorem (similarly as in the proof of Proposition~\ref{thm_first_mom}) and get
    \begin{align*}
       \lim\limits_{n\to\infty} n^{1+\delta'} 	\int\limits_{\varepsilon}^{\infty} \lb \E [ \xi \exp(-s\xi^2)] \rb^2 \varphi^{n-2} (s) ds  = 0. 
    \end{align*}
It remains to show that for $\delta'< \delta$ the integral 
	\begin{align*}
	   E_n (\varepsilon) =  \int\limits_0^\varepsilon n^{1+\delta'}
	    \lb \E \left[ \xi \exp(-s\xi^2)  \right] \rb^2 \varphi^{n-2} (s) ds
	\end{align*}
	converges to zero, as $n$ tends to infinity. 
	Let $s\in(0,\varepsilon).$
	We use the following identity given on page 1525 of \cite{gine:goetze:mason:1997},
	\begin{align*}
	    \E \left[ \xi \exp(-s\xi^2)  \right] = - s \E \left[ \xi^3 \exp(-s \theta \xi^2)  \right]
	\end{align*}
	for some $\theta $ uniformly distributed on the interval $[0,1]$ and independent of $\xi$.  This implies
	\begin{align*}
	    \lb  \E \left[ \xi \exp(-s\xi^2)  \right] \rb^2 
	    &=  s^2 \lb \E \left[ \xi^3 \exp(-s \theta \xi^2)  \right] \rb^2 
	    \leq  s^2 \lb \E \left[ | \xi |^{1+\delta} |\xi|^{2-\delta} \exp(-s \theta \xi^2)  \right] \rb^2 \\
	    &\lesssim  s^2 \lb s^{-1/2 (2-\delta)} \rb^2 \lb
	    \E | \xi|^{1+\delta} \E [ \theta^{-1/2(2-\delta)} ]\rb^2 
	    \lesssim  s^{\delta},
	\end{align*}
 where we optimized over $\xi$ for the second to last inequality. 
	Using also that
	(similarly to \eqref{ineq_nsphi}) 
	    \begin{align*}
	        \lb n s \rb ^{1+\delta'} \varphi^{n - 2} (s) \lesssim 1, ~ s \in (0,\varepsilon),
	    \end{align*}
	    we get
	    \begin{align*}
	        n^{1+\delta'}
	    \lb \E \left[ \xi \exp(-s\xi^2)  \right] \rb^2 \varphi^{n-2} (s)
	    =  \lb n s \rb ^{1+\delta'} \varphi^{n-2} (s)
	    \lb \E \left[ \xi \exp(-s\xi^2)  \right] \rb^2 
	    s^{-1-\delta'}
	    \lesssim s^{- 1 - \delta' + \delta},
	    \end{align*}
	    which is integrable on $(0,\varepsilon)$ since $\delta' < \delta$.
	       As a result, we are allowed to apply the dominated convergence theorem and the assertion $E_{n}(\varepsilon)= o(1)$ follows from $n^{1+\delta'}\varphi^{n-2}(s) = o(1)$, which is a consequence of Lemma 3.1 in \cite{fuchs:joffe:teugels:2001}.
\end{proof}

By Corollary 2 in \cite{elkaroui:2009}, we have the following lemma.
\begin{lemma}\label{lem:expectedtransform1}
It holds
\begin{equation*}
s_n(z) - \E[s_n(z)] \cas 0\,, \qquad \nto\,, z\in \mathbb{C}^+.
\end{equation*}
\end{lemma}


For convenience, the following lemma is formulated for the matrix $\bfY$ defined in \eqref{def_yij}. However, its proof reveals that it holds more generally for any $p\times n$ random matrix with i.i.d. rows. 
\begin{lemma}\label{lemma3.2}
For the $p\times n$ random matrix $\Y$, it holds
\begin{equation*}
-z \E[s_{\Y\Y^\top}(z)]= \E\left[ \frac{1}{1+\ty_1 (\bfY_{-1}^\top\bfY_{-1} -z \bfI_n)^{-1} \ty_1^\top} \right]\,, \quad z\in\mathbb{C}^+.
\end{equation*} 
\end{lemma}
\begin{proof}
By Theorem A.4 in \cite{bai:silverstein:2010}, we have
\begin{equation*}
s_{\Y\Y^\top}(z)= \frac{1}{p} \sum_{k=1}^p \frac{1}{\ty_k \ty_k^\top-z-\ty_k \Y_{-k}^\top (\Y_{-k}\Y_{-k}^\top-z\bfI_{p-1})^{-1} \Y_{-k} \ty_k^\top}\,,
\end{equation*}
where $\ty_k\in \R^{1\times n}$ are the rows of $\Y$ and $\Y_{-k}\in \R^{(p-1)\times n}$ is the matrix $\Y$ with the $k$th row removed. 
Since the rows of $\bfY$ are i.i.d., we obtain
\begin{equation*}
\E[s_{\Y\Y^\top}(z)] = \E \left[\frac{1}{\ty_1 \ty_1^\top-z-\ty_1 \Y_{-1}^\top (\Y_{-1}\Y_{-1}^\top-z\bfI_{p-1})^{-1} \Y_{-1} \ty_1^\top} \right]\,.
\end{equation*}
An application of the identity $\Y_{-1}^\top (\Y_{-1}\Y_{-1}^\top-z\bfI_{p-1})^{-1} \Y_{-1}= \bfI_n + z (\bfY_{-1}^\top\bfY_{-1} -z \bfI_n)^{-1}$ finishes the proof of the lemma.
\end{proof}
The following result is needed for the analysis of $W_n(z)$.
\begin{lemma}\label{le:eyy} 
Assume $p/n\to \gamma$, $p,n\to\infty$ and $\E[Y_{11} Y_{12}]=o(n^{-1})$.
Then we have for $z\in\mathbb{C}^+$
\begin{equation}
\E \big[\ty_1 \bfD(z) \ty_1^\top |\Y_{-1} \big]- \frac{1}{n} \tr \big( \bfD(z) \big)  \cas  0\,, \quad \nto\,.
\end{equation}
\end{lemma}
\begin{proof}
Note that $\E \big[\ty_1 \bfD(z) \ty_1^\top |\Y_{-1} \big]= \tr\big( \bfD(z) \E[\ty_1^\top\ty_1])$ and $\E[\ty_1^\top\ty_1]= (n^{-1} -\E[Y_{11}Y_{12}]) \bfI_n +\E[Y_{11}Y_{12}] 1_n 1_n^\top$ with $1_n=(1,\ldots,1)^\top\in \R^n$.
Hence, we have
\begin{equation*}
\E \big[\ty_1 \bfD(z) \ty_1^\top |\Y_{-1} \big]- \frac{1}{n} \tr \big( \bfD(z) \big)= -\E[Y_{11}Y_{12}] \tr \big( \bfD(z) \big) + \E[Y_{11}Y_{12}] \, 1_n^\top \bfD(z) 1_n\,.
\end{equation*}
Because of $|\tr \big( \bfD(z) \big)| \le n / \Im(z)$ (see \eqref{eq:le1}) and $\E[Y_{11}Y_{12}]=o(n^{-1})$, it remains to prove that 
\begin{equation}\label{eq:dgerts}
\E[Y_{11}Y_{12}] \, 1_n^\top \bfD(z) 1_n \cas 0\,, \quad \nto.
\end{equation}
To this end, we have in view of \eqref{est_quad_form} and \eqref{eq:le1} that
$1_n^\top \bfD(z) 1_n \leq n  \norm{\bfD(z)} \leq n/\Im(z)$,
which implies \eqref{eq:dgerts}.

\end{proof}

	\section{Proof of Theorem \ref{thm}} \label{sec_proof_dep}
		
	Our main result for the dependent case immediately follows from the two subsequent lemmas, which are proven in Section \ref{sec_proof_lem1} and \ref{sec_proof_lem2}, respectively. The strategy for these proofs was discussed in Section \ref{sec_strat_proof}. 
	Throughout this section, we work in the setting of Section \ref{sec_dependent} and, in particular, under the assumptions of Theorem \ref{thm} if not explicitly stated otherwise.

	\begin{lemma}\label{lem_step1}
	 	For all $z\in\mathbb{C}^+$, we have almost surely
	 	\begin{align*}
	 		\lim\limits_{n\to\infty} \lb s_n(z) - \E [ s_n(z) ] \rb  =0.
	 	\end{align*}
	\end{lemma}
	
	\begin{lemma} \label{lem_step2}
		For $z\in\mathbb{C}^+$, we have
		\begin{align*}
			\lim\limits_{n\to\infty} \E [ s_n(z) ] = s(z),
		\end{align*}
		where $s(z)$ is the unique solution to \eqref{eq_s}. 
	\end{lemma}
	
	In order to prove Lemma \ref{lem_step1} and Lemma \ref{lem_step2}, we need some preparation. 
For $0 \leq j \leq n$, let $\E_j$ denote the conditional expectation with respect to $\tilde{\bfx}_1, \ldots, \tilde{\bfx}_j$. 
	Then, $\E_0 [Z] = \E[Z]$ and $\E_n[ Z ] =  Z $ for some random variable Z which is measurable with respect to the $\sigma$-field generated by $\tilde{\bfx}_1, \ldots, \tilde{\bfx}_n$.
	
For $j\in\{1, \ldots, n\}$ and $z\in\mathbb{C}^+$ recall that $\mathbf{y}_j = (Y_{1j}, \ldots, Y_{pj} ) ^\top$,  $\mathbf{x}_j = (X_{1j}, \ldots, X_{pj} ) ^\top$ and define
		 \begin{align*}
		 	\D(z) &=  \bfR - z \mathbf{I} , \qquad
		 	\D_j(z) =  \D(z) - \mathbf{y}_j \mathbf{y}_j^\top ,
			\qquad \mathbf{S}^{(j)} =   \bfS - n^{-1} \bfx_j \bfx_j^\top ,\\
		 	\beta_j(z) &=  \frac{1}{1 + \bfy_j^\top \D_j^{-1}(z) \bfy_j } ,
		 ~~\qquad\qquad	\overline{\beta}_j(z) =  \frac{1}{1 + n\invv \tr \T \D_j^{-1}(z)  } ,\\
	b(z) &=  \frac{1}{1 +n\invv \tr \E[ \T \D^{-1}(z) ] }, 
		 \qquad	\K (z) =  b(z) \T . 
		 \end{align*}
		 Let $j\in\{1, \ldots, n\}$. Note that $\D_j(z)$ is not independent of $\mathbf{y}_j$. For later considerations, we aim to define an appropriate approximate $\hat{\D}_j(z)$ which is independent of $\mathbf{y}_j$. For this purpose, define the diagonal matrices $\mathbf{M}^{(j)}$ and $\mathbf{M}$ with entries 
		\begin{align*}			 
		 M^{(j)}_{ii} = &   \Bigg( \frac{1}{n} \sum_{k=1; k\neq j}^n X_{ik}^2 \Bigg)^{-1/2} \quad \text{ and } \quad
		 M_{ii} =    \left( \frac{1}{n} \sum\limits_{k=1}^n X_{ik}^2\right)^{-1/2}\,, \quad ~1 \leq i \leq p, 
		 \end{align*}
		respectively, and set 
		 $ \mathbf{Y}^{(j)} = ( \mathbf{y}^{(j)}_1, \ldots, \mathbf{y}^{(j)}_n) 
		 = n^{-1/2} \mathbf{M}^{(j)} \X$. 
		 Note that we also have 
		 $ \mathbf{Y} = ( \mathbf{y}_1, \ldots, \mathbf{y}_n) 
		 = n^{-1/2} \mathbf{M} \X$.
		 Then, the matrix 
		 $$ \hat{\D}_j (z) =
	 \Y^{(j)} \lb \Y^{(j)}\rb ^\top - \mathbf{y}^{(j)}_j \lb \mathbf{y}^{(j)}_j \rb^\top - z \mathbf{I}$$
		 is independent of $\mathbf{y}_j$. 
		We also set
$\tilde{\D}_j (z) = \mathbf{Y}^{(j)} \lb \mathbf{Y}^{(j)}  \rb^\top - z\mathbf{I}$.

	\subsection{Proof of Lemma \ref{lem_step1}} \label{sec_proof_lem1}

	Noting that $(\E_j - \E_{j-1}) \tr \hat{\D}_j\invv= 0$, $1 \leq j \leq n$,
	we decompose 
	\begin{align*}
	 s_n(z) - \E [ s_n(z) ] 
	& = \frac{1}{p}  \lb \tr  \D\invv(z) - \E [ \tr \D\invv(z) ] \rb \\
	& = \frac{1}{p} \sum\limits_{j=1}^n ( \E_j - \E_{j - 1}) \tr \D\invv(z)
	= T_{1,n} + T_{2,n},
	\end{align*}
	where 
	\begin{align*}
		T_{1,n} =& \frac{1}{p} \sum\limits_{j=1}^n ( \E_j - \E_{j - 1}) \tr \lb \tilde{\D}_j \invv(z) - \hat{\D}_j\invv(z) \rb , \\
		T_{2,n} = & \frac{1}{p} \sum\limits_{j=1}^n ( \E_j - \E_{j - 1}) \tr \lb  \D\invv(z) - \tilde{\D}_j\invv(z) \rb.
	\end{align*}
	First, consider the random variable $T_{1,n}$ and use Lemma 2.6 in \cite{silverstein:bai:1995} to obtain
	\begin{align*}
		\frac{1}{p}\left| ( \E_j - \E_{j - 1}) \tr \lb \tilde{\D}_j \invv(z) - \hat{\D}_j\invv(z) \rb \right| \leq \frac{2}{p v},
	\end{align*}
	where $v= \operatorname{Im}(z)>0$. 
	Similarly as in the proof of Lemma 6 in \cite{elkaroui:2009}, we conclude for $\varepsilon >0$ by invoking Azuma's inequality for real and imaginary parts that
	\begin{align*}
	\mathbb{P} \lb | T_{1,n} | > \varepsilon \rb  
	\leq 4 \exp \lb - \frac{\varepsilon^2 p^2 v^2}{16 n} \rb .
	\end{align*}
	By the Borel-Cantelli lemma, this implies the almost sure convergence
$T_{1,n} \cas 0$ for $n\to\infty$.
\medskip
	
		Investigating the random variable $T_{2,n}$ further, we write 
		\begin{align*}
			\D\inv(z) - \tilde{\D}_j\inv (z) 
			= \lb \mathbf{MSM} - z\mathbf{I} \rb\inv 
			- \lb \mathbf{M}^{(j)} \rb \inv \mathbf{M} \lb \mathbf{MSM} - z \mathbf{M}^2 \lb  \bfM^{(j)}  \rb^{-2} \rb \inv \mathbf{M} \lb \mathbf{M}^{(j)} \rb\inv .
		\end{align*}
and decompose 
	 	$\tr \lb \D\inv(z) - \tilde{\D}_j\inv (z)  \rb 
	 	= Q_{j,1} + Q_{j,2}$,
	 where
	 \begin{align*}
	 	Q_{j,1} = & \tr \lb \lb \mathbf{MSM} - z\mathbf{I} \rb\inv 
			- \lb \mathbf{MSM} - z \mathbf{M}^2 \lb \mathbf{M}^{(j)} \rb^{-2} \rb \inv   \rb \\
			= & z \tr \lb 
			\lb \mathbf{MSM} - z \mathbf{M}^2 \lb \mathbf{M}^{(j)} \rb^{-2} \rb\inv
			 \lb \mathbf{I} -  \mathbf{M}^2 \lb \mathbf{M}^{(j)} \rb^{-2}  \rb 
			\lb \mathbf{MSM} - z\mathbf{I} \rb\inv 
			\rb 
			, \\
			Q_{j,2} =& \tr \lb 
			  \lb \mathbf{MSM} - z \mathbf{M}^2 \lb  \mathbf{M}^{(j)}  \rb^{-2} \rb \inv  \lb \mathbf{I} - \mathbf{M}^2 \lb \mathbf{M}^{(j)} \rb ^{-2} \rb 
			 \rb.
	 \end{align*}
	 Using that the diagonal matrix $\mathbf{I} - \mathbf{M}^2 \lb \mathbf{M}^{(j)} \rb^{-2} $ is nonnegative definite, since 
	 \begin{align*}
	 	\lb \mathbf{M}^2 \lb \mathbf{M}^{(j)} \rb^{-2} \rb_{ii} 
	 	= 1 - \frac{X_{ij}^2 }{\sum_{k=1}^n X_{ik}^2}
	 	= 1 - Y_{ij}^2, ~\quad 1 \leq i \leq p,
	 \end{align*}
	we estimate
	\begin{align*}
		| Q_{j,1}  | \lesssim \tr \lb \mathbf{I} - \mathbf{M}^2 \lb \mathbf{M}^{(j)} \rb ^{-2} \rb = \sum\limits_{i=1}^p Y_{ij}^2 
\end{align*}		
	and a similar bound can be shown for the term $Q_{j,2}$. 
Therefore, we deduce that 
	\begin{align} \label{mom_bound}
		\E \left| \tr \lb \D\inv(z) - \tilde{\D}_j\inv (z)  \rb \right|^2 \lesssim  \E\Big|\sum_{i=1}^p Y_{ij}^2\Big|^2
		= \sum_{i,k=1}^p \E [Y_{kj}^2 Y_{ij}^2 ]
		\lesssim 1\,,
	\end{align}
	where the last estimate holds due to Lemma \ref{lem_mixed_second_mom}. 
By an application of Burkholder's inequality for some $q>2$, this implies
	\begin{align*}
	\E \left| T_{2,n} \right|^q \lesssim p^{-q} \E \lb \sum\limits_{j=1}^n \left| (\E_j - \E_{j-1} ) \tr \lb \D\inv(z) - \tilde{\D}_j\inv (z)  \rb \right|^2 \rb^{\frac{q}{2}} 
	\lesssim n^{\frac{q}{2} } p^{-q} = o(1). 
	\end{align*}
Again by Borel-Cantelli, we get $T_{2,n}\cas 0$ as $\nto$. 	Eventually, the assertion of Lemma \ref{lem_step1} follows. 
	
	\subsection{Proof of Lemma \ref{lem_step2}} \label{sec_proof_lem2}
	
		The fact that the solution of the equation \eqref{eq_s} is uniquely determined follows from \cite{bai:zhou:2008}. \\ 
We see that
	\begin{align} \label{a1}
		\D(z) - \lb \K(z) - z\mathbf{I} \rb = \sum\limits_{j=1}^n \bfy_j \bfy_j^\top - \K(z).
	\end{align}
	Note that, by using formula (6.1.11) in \cite{bai:silverstein:2010}, we have for $1 \leq j \leq n$
	\begin{align} \label{a51}
		\bfy_j^\top \D^{-1}(z) 
		= \beta_{j}(z) \bfy_j^\top \D_{j}^{-1}(z)  .
	\end{align}
	We multiply equation \eqref{a1} with $( \mathbf{K}(z) - z \mathbf{I})\invv$ from the left and $\D\invv(z)$ from the right and use \eqref{a51} to obtain
	\begin{align*}
	& \lb \mathbf{K}(z) - z \mathbf{I} \rb \invv - \D\invv(z) \\
	= &  \sum\limits_{j=1}^{n} \lb \mathbf{K}(z) - z \mathbf{I}\rb\invv \rd_j \rd_j^\top \D\invv(z) 
	-  \lb \mathbf{K}(z) - z \mathbf{I}\rb\invv \mathbf{K}(z) \D\invv(z) \\
	= & \sum\limits_{j=1}^{n} \beta_{j}(z) \lb \mathbf{K}(z) - z \mathbf{I}\rb\invv \rd_j \rd_j^\top \D_{j}\invv(z)
	- \lb \mathbf{K}(z) - z \mathbf{I}\rb\invv \mathbf{K}(z) \D\invv(z). \\
	\end{align*}
	This implies for $l\in\{0,1\}$ 
	\begin{align*}
		& \T_n^l \lb \mathbf{K}(z) - z \mathbf{I} \rb \invv - \T_n^l \D\invv(z) \\
		=& \sum\limits_{j=1}^{n} \beta_{j}(z) \T_n^l \lb   \mathbf{K}(z) - z\mathbf{I} \rb\invv \rd_j \rd_j^\top \D_{j}\invv(z) - \T_n^l \lb \mathbf{K}(z) - z\mathbf{I} \rb \invv \mathbf{K}(z) \D\invv(z).
	\end{align*}
	Taking traces and dividing by $p$, we conclude
	\begin{align*}
		& \frac{1}{p} \tr  \T_n^l \lb  \mathbf{K}(z) - z \mathbf{I} \rb \invv 
		- \frac{1}{p} \tr \T_n^l \D\invv(z) \\
		= & \frac{1}{p}\sum\limits_{j=1}^{n} \beta_{j}(z) \rd_j^\top \D_{j}\invv(z) \T_n^l \lb \mathbf{K}(z) - z\mathbf{I} \rb \invv \rd_j 
		 - \frac{1}{p} \tr \T_n^l \lb \mathbf{K}(z) - z\mathbf{I} \rb\invv \mathbf{K}(z) \D\invv(z) \\
		 = & \frac{1}{p} \sum\limits_{j=1}^{n} \beta_{j}(z) \varepsilon_j,
	\end{align*}
	where 
	\begin{align*}
		\varepsilon_j
		= & \rd_j^\top \D_{j}\invv(z) \T_n^l \lb \mathbf{K}(z) - z\mathbf{I} \rb\invv \rd_j 
		 - n\invv \beta_{j}\invv(z) \tr \T_n^l \lb \mathbf{K}(z) - z\mathbf{I} \rb \invv \mathbf{K}(z) \D\invv(z) 
		\\
		= &   \rd_j^\top \D_{j}\invv(z) \T_n^l \lb \mathbf{K}(z) - z\mathbf{I} \rb\invv \rd_j
		 - n\invv \tr \T_n^l \lb \mathbf{K}(z) - z\mathbf{I} \rb \invv \mathbf{K}(z) \D\invv(z) 
		 \lb 1 + \rd_j^\top \D_{j}\invv(z) \rd_j \rb.
	\end{align*}
	We decompose $\varepsilon_j = \varepsilon_{j1} + \varepsilon_{j2} + \varepsilon_{j3}$, where
	\begin{align*}
		\varepsilon_{j1} = &
		n\invv \tr \T_n^{l+1} \lb  \mathbf{K}(z) - z \mathbf{I} \rb \invv \D_{j}\invv (z) 
		- n\invv \tr \T_n^{l+1} \lb  \mathbf{K}(z) - z \mathbf{I} \rb\invv \D\invv (z) \\
		\varepsilon_{j2} = &
		\rd_j^\top \D_{j}\invv(z) \T_n^l \lb  \mathbf{K}(z) - z \mathbf{I} \rb\invv \rd_j
		- n\invv \tr \T_n^{l+1} \lb  \mathbf{K}(z) - z \mathbf{I} \rb\invv \D_{j}\invv(z) \\
		\varepsilon_{j3} = & 
		 - n\invv \tr \T_n^l \lb \mathbf{K}(z) - z\mathbf{I} \rb \invv \mathbf{K}(z) \D\invv(z) 
		 \lb 1 + \rd_j^\top \D_{j}\invv(z) \rd_j \rb 
		 +n\invv \tr \T_n^{l+1} \lb  \mathbf{K}(z) - z \mathbf{I} \rb\invv \D\invv (z) \\
		 = & - n\invv \tr \T_n^{l+1} \lb \mathbf{K}(z) - z \mathbf{I} \rb\invv \D\invv(z) 
		 \left\{ b(z)\lb \rd_j^\top \D_{j}\invv(z) \rd_j + 1\rb- 1 \right\},
	\end{align*}
	using that $\T_n$ and $( \mathbf{K}(z) - z\mathbf{I} ) \invv$ commute.	
	
	In the following, we will show that for $n\to\infty$
	\begin{align} \label{epsilon_conv}
	\E \left[ \frac{1}{p} \sum\limits_{j=1}^{n} \beta_{j}(z) \varepsilon_{j,r} \right]
	\to 0, ~\quad r\in\{1,2,3\}.  
\end{align}		
	 Similar arguments as given by \cite{bai:silverstein:2010} for their estimate (9.9.13)
yield
	$$ \| (\K - z \mathbf{I} )\invv \| \lesssim 1.$$ 
 For the term $\varepsilon_{j2}$, we substitute the matrix $\D_j\inv(z)$ by $\hat{\D}_j\inv(z)$ resulting in an asymptotically negligible error by Lemma \ref{lem_Dhat}. To be precise, we have for any $p\times p$ matrix  $\mathbf{A}$ independent of $\rd_j$ with bounded spectral norm
	\begin{align} \label{substitute1}
		\E | \rd_j^\top \D_j\inv(z) \mathbf{A} \rd_j - \rd_j^\top \hat{\D}_j\inv(z) \mathbf{A} \rd_j | 
		\lesssim \lb \E \|\D_j\inv(z) -\hat{\D}_j\inv(z) \|^2 
		\E (\rd_j^\top \rd_j )^2 
		\rb\sq =o(1) ,
	\end{align}
	and 
	\begin{align}
		n\inv \E \left| \tr \lb \D_j\inv(z) - \hat{\D}_j\inv(z) \rb  \mathbf{A} \right|
		\lesssim \| \D_j\inv(z) - \hat{\D}_j\inv(z) \| = o(1), \label{subsitute2}
	\end{align}
	where we used Lemma \ref{lem_mixed_second_mom} and Lemma \ref{lem_Dhat}. 
	Note that the matrix $\hat{\D}_j\inv(z)$ is independent of $\mathbf{y}_j$. 
	Consequently, we obtain the desired result for $\varepsilon_{j2}$ from Lemma \ref{lem_quad_form}. 
	Regarding $\varepsilon_{j1}$, we proceed similarly as in \cite{bai:zhou:2008} and apply Lemma 2.6 of \cite{silverstein:bai:1995}. 
	For $\varepsilon_{j3},$ we note that 
	\begin{align*}
		| \varepsilon_{j3} |
		&=  \left| n\invv \tr \T_n^{l+1} \lb  \mathbf{K} - z \mathbf{I} \rb\invv \D\invv(z) 
		 \left\{ b(z)\lb \rd_j^\top \D_{j}\invv(z) \rd_j 
		 - \E \left[ n\invv \tr \T_n \D\invv(z) \right] \rb \right\} \right| \\
		& \lesssim   | \varepsilon_{j31} | + | \varepsilon_{j32} | + | \varepsilon_{j33} |  ,
	\end{align*}
	where 
	\begin{align*}
		\varepsilon_{j31} &=  \rd_j^\top \D_{j}\invv(z) \rd_j - n\invv \tr \T_n \D_{j}\invv(z), \\
		\varepsilon_{j32} &=  n\invv \tr \T_n \D_{j}\invv(z) - n\invv \tr \T_n \D\invv(z)
		= n\invv \beta_j(z) \rd_j^\top \D_j\invv (z) \T_n \D_j\invv (z) \rd_j
		, \\
		\varepsilon_{j33} &=  n\invv \tr \T_n \D\invv(z) - \E \left[ n\invv \tr \T_n \D\invv(z) \right] .
	\end{align*}
	The term $\varepsilon_{j31}$ can be treated similarly to $\varepsilon_{j2}$
	and $\varepsilon_{j33}$ similarly to Lemma \ref{lem_step1}. 
	For the remaining term $\varepsilon_{j32}$, we obtain
	\begin{align*}
		\E | \varepsilon_{j32} | 
		\lesssim n\inv \E [ \rd_j^\top \rd_j ] = n\inv \sum\limits_{i=1}^p \E [ Y_{ij}^2]  =\frac{p}{n^2}= o(1) .
	\end{align*}
	Thus, the convergence in \eqref{epsilon_conv} holds true, which implies for $l \in \{0,1\}$
	\begin{align} \label{conv}
	\frac{1}{p} \lb \E \tr  \T_n^l \lb \mathbf{K}(z) - z \mathbf{I} \rb \invv 
		- \E \tr \T_n^l \D\invv(z) \rb \to 0, ~\quad n\to\infty.
	\end{align}
	Using \eqref{conv} with $l=0$ and $l=1$, we have for $n\to\infty$,
	\begin{align} 
		& \frac{1}{p} \E \tr \lb \frac{\T_n}{1 + \gamma_n a_n(z) } - z\mathbf{I}\rb\invv  - \E[s_n(z)] \to 0, \label{conv1} \\
		& \frac{1}{p} \E \tr \T_n \lb \frac{\T_n }{1 + \gamma_n a_n(z) } - z \mathbf{I}\rb\invv  - a_n(z) \to 0,
		\label{conv2}
	\end{align}
	where $a_n(z) = p\invv \E [ \tr \T_n \D_j\invv(z)  ]$ and $\gamma_n=p/n$. 
Using 
$|1/(1 + \gamma_n a_n(z)) | \lesssim 1$,
	we conclude from \eqref{conv2} that
	\begin{align*}
		1 + \frac{z}{p} \E \tr \lb \frac{  \T_n}{1 + \gamma_n a_{n}(z) } - z \mathbf{I} \rb \invv
		- \frac{a_{n}(z)}{1 + \gamma_n a_{n}(z)} \to 0.
	\end{align*}
 	Combining this with \eqref{conv1} yields
	\begin{align*}
		1 + z \E [ s_{n}(z) ] - \frac{a_{n}(z)}{1 + \gamma_n a_{n}(z)} \to 0
	\end{align*}
	and, by rearranging terms and multiplying with $\gamma_n$,
	\begin{align*}
		\frac{1}{1 + \gamma_n a_{n}(z) }  
		= 1 - \gamma_n ( 1 + z \E [s_n(z)]) + o(1).
	\end{align*}
	Substituting this in \eqref{conv1}, we get
	\begin{align} \label{approx_eq}
		\frac{1}{p} \E \tr \lb \T_n \lb 1 - \gamma_n ( 1 + z \E [ s_n(z) ] ) \rb  - z \mathbf{I}\rb\invv - \E [ s_n(z) ] \to 0.
	\end{align}
	Note that $(\E [s_n(z)])_{n\in\N}$ is a bounded sequence for any fixed $z\in\C^+$, so that, by the Bolzano-Weierstraß theorem, each subsequence of $(\E [s_n(z)])_{n\in\N}$ contains a converging subsequence. 
	It is left to show the uniqueness of the limit. Let $(\E[s_{k(n)}(z)])_{n\in\N}$ and $(\E[s_{j(n)}(z)])_{n\in\N}$ be two subsequences of $(\E [s_n(z)])_{n\in\N}$ which converge to $m_1(z)$ and $m_2(z)$, respectively. Using \eqref{approx_eq}, we see that both $m_1(z)$ and $m_2(z)$ satisfy \eqref{eq_s}. 
	Since the solution to \eqref{eq_s} is unique as discussed at the beginning of this proof, we have $m_1(z) = m_2(z)$ and finally conclude that 
	\begin{align*}
		\lim\limits_{n\to \infty} \E [ s_{n} (z) ] = s(z),
	\end{align*}
	where $s(z)$ satisfies \eqref{eq_s}. This finishes the proof of Lemma \ref{lem_step2}.

	
	\subsection{Quadratic forms} \label{sec_quad_forms}
	The crucial step in the proof of Lemma \ref{lem_step2} relies on the concentration of quadratic forms in $\rd_j$, which is the content of the following lemma. A proof can be found in Appendix \ref{appendix_c}. 
	
	\begin{lemma}\label{lem_quad_form}
For $j\in\{1, \ldots, n\}$ and $n\in\N$, let $\mathbf{B}^{(j,n)} = \mathbf{B}^{(j)} = (B_{ik}^{(j)})_{1 \leq i,k \leq p} \in\mathbb{C}^{p\times p}$ be matrices independent of $\rd_j$, which satisfiy
	$$ 
	\sup\limits_{n\in\N} \sup\limits_{1 \leq j \leq n} \norm{ \mathbf{B}^{(j)}} < \infty. 
	$$
	 Then,
		\begin{align*}
			V_n = \E \left|  \rd_j^\top \mathbf{B}^{(j)} \rd_j - n\invv \tr \T \mathbf{B}^{(j)}  \right|^2 \to 0,
		\end{align*}
		as $n\to \infty$ uniformly in $j\in\{1, \ldots, n\}$. 
	\end{lemma}	
	The assumption that the matrix $\bfB^{(j)}$ is independent of $\bfy_j$ is crucial for the proof of Lemma~\ref{lem_quad_form}. However, when considering the proof of Lemma \ref{lem_step1}, the $\bfB^{(j)}$ involves the resolvent $\D_j\inv(z)$ which violated the independence assumption due to the complex dependence structure of $\bfY$ in both rows and columns.
	The following lemma shows us how to overcome this obstacle and allows us to approximate the resolvent by a matrix independent of $\bfy_j$. 
		\begin{lemma} \label{lem_Dhat}
	It holds for all $q\geq 1$, $j\in\{1, \ldots, n\}$
	\begin{align*}
		\E \| \D_1\invv(z) - \hat{\D}_1\invv(z) \|^q  = \E \| \D_j\invv(z) - \hat{\D}_j\invv(z) \|^q \to 0,
	\end{align*}
	as $n\to\infty$. 
	\end{lemma}
	
	\begin{proof}[Proof of Lemma \ref{lem_Dhat}]
	We first note that it is sufficient to show convergence in probability, since the random variable $\| \D_j\invv(z) - \hat{\D}_j\invv(z) \|$ is bounded uniformly in $n\in\N$. 
	As a preparation, we derive that 
	\begin{align}
	& \| \mathbf{M} \lb \mathbf{ M}^{(j)} \rb\inv - \mathbf{I} \| \conp 0. \label{con1}
	\end{align}
	For this purpose, we note that 
	\begin{align*}
		\| \bfM \lb \Mj \rb\inv - \bfI \| 
		=& \max_{1 \leq i \leq p} 
		\left| 
		\sqrt{1 - \frac{X_{ij}^2 }{\sum_{t=1}^n X_{it}^2}}
		- 1 \right| 
		\leq \max\limits_{1 \leq i \leq p} 
		 \frac{X_{ij}^2 }{\sum\limits_{\substack{t=1}}^n X_{it}^2} \\
		\leq & \max_{1 \leq i \leq p} 
		 \frac{X_{ij}^2 }{n}
		 \left(\min_{1\leq i \leq p} \frac{1}{n} \sum_{t=1}^n X_{it}^2 \right)^{-1}\,.
	\end{align*}
	By assumptions \ref{a_eigen_T} and \ref{a_2mom}, we have that
$\max_{1 \leq i \leq p} 
		 n^{-1} X_{ij}^2  \conp 0$
	and it follows from Theorem 1 in \cite{tikhomirov:2015}, 
	\begin{align*}
		\lambda_{\min} \lb \frac{1}{n} \tilde{\bfX} \tilde{\bfX}^\top \rb 
		\to \lb 1 - \sqrt{\gamma} \rb ^2 
		\textnormal{ almost surely. }
	\end{align*}
	Hence, we conclude using assumption \ref{a_eigen_T}
	\begin{align*}
		\min\limits_{1\leq i \leq p} \frac{1}{n} \sum\limits_{\substack{t=1}}^n X_{it}^2
		\geq & \lambda_{\min} (\T) \lambda_{\min} \lb \frac{1}{n } \tilde{\mathbf{X}} \tilde{\mathbf{X}}^\top \rb >\eta \textnormal{ almost surely},
	\end{align*}
	for some $\eta>0$ and hence, \eqref{con1} holds true. In order to show that $\| \D_j\invv(z) - \hat{\D}_j\invv(z) \|=o_{\PR}(1),$  we will approximate the resolvent $\D_j\inv(z)$ by an appropriate matrix. 
	Using 
	$$\D_j\inv(z) = \mathbf{M}\inv \lb \mathbf{S}^{(j)}  - z \bfM^{-2} \rb \inv \bfM\inv,$$
	 we write
	\begin{align*}
		\D_j\inv(z) - \lb \bfI - \lb \Mj \rb\inv \bfM \rb \D_j\inv(z) 
		= \lb \Mj \rb\inv \lb \mathbf{S}^{(j)}  - z \bfM^{-2} \rb \inv \bfM\inv
	\end{align*}
	and
	\begin{align*}
		& \lb \Mj \rb\inv \lb \mathbf{S}^{(j)}  - z \bfM^{-2} \rb \inv \bfM\inv
		- \lb \Mj\rb\inv \lb \mathbf{S}^{(j)}  - z \bfM^{-2} \rb\inv \bfM\inv 
		\lb \bfI - \bfM \lb \Mj\rb\inv \rb \\
		= & \lb \Mj \bfS^{(j)} \Mj - z \bfM^{-2} \lb \Mj\rb^2 \rb\inv .
	\end{align*}
	By \eqref{con1}, we conclude that
	\begin{align*}
		\left\| \lb \bfI - \lb \Mj \rb\inv \bfM \rb \D_j\inv(z) \right\|
		& \conp 0, \\
		 \left\| \lb \Mj\rb\inv \lb \mathbf{S}^{(j)}  - z \bfM^{-2} \rb\inv \bfM\inv 
		\lb \bfI - \bfM \lb \Mj\rb\inv \rb  \right\| & \conp 0. 
	\end{align*}
	Thus, since $\mathbf{A}\inv - \mathbf{B}\inv = \mathbf{B}\inv ( \mathbf{B} - \mathbf{A}) \mathbf{A}\inv$ for nonsingular matrices $\mathbf{A}, \mathbf{B}$, we obtain 
	\begin{align*}
		&\| \D_j\inv(z)  - \hat{\D}_j\inv(z) \| \leq  
		\left\|  \lb \Mj \bfS^{(j)} \Mj - z \bfM^{-2} \lb \Mj\rb^2 \rb\inv 
		- \hat{\D}_j\inv(z) 
		\right\| + o_{\mathbb{P}} (1) \\
		\leq & \left\| \hat{\D}_j\inv(z) \right\| 
		\left\|\lb \Mj \bfS^{(j)}  \Mj - z \bfM^{-2} \lb \Mj\rb^2 \rb\inv  \right\|
		\left\| z \lb \bfI - \bfM^{-2} \lb \Mj\rb^2 \rb  \right\| 	
		+ o_{\mathbb{P}} (1) \\
		\lesssim & 
		|z | \left\|   \bfI - \bfM^{-2} \lb \Mj\rb^2  \right\| 	
		+ o_{\mathbb{P}} (1) =  o_{\mathbb{P}} (1)\,, \qquad \nto\,.
	\end{align*}
	Here, it can be shown similarly to \eqref{con1} that the  term in the last line is asymptotically negligible. 
	\end{proof}
	
	\subsection{Moments of $\mathbf{Y}$ }	\label{sec_mom_y}

	To begin with, we formulate a consequence of the proof of Lemma \ref{lem_quad_form} about mixed second moments of entries of $\mathbf{Y}$ belonging to the same column. 
		\begin{lemma} \label{lem_mixed_second_mom}
	It holds for all $j\in\{1, \ldots, n\}$
		\begin{align*} 	
	\sum\limits_{i,k=1}^p \E [Y_{k1}^2 Y_{i1}^2 ]  = 
	\sum\limits_{i,k=1}^p \E [Y_{kj}^2 Y_{ij}^2 ] \lesssim 1.
	\end{align*}
	\end{lemma}	
	\begin{proof}
	Note that, as $\nto$,
	\begin{align*}
		\sum\limits_{i,k=1}^p \E [Y_{kj}^2 Y_{ij}^2 ] 
		= & \sum\limits_{k=1}^p \E [ Y_{kj}^4] 
		+ \tilde{V}_{n,2} + \frac{p(p-1)}{n^2} \lesssim 1,
	\end{align*}
	where the fact that 
	\begin{align*}
	\tilde{V}_{n,2} := \sum\limits_{\substack{i,k=1, \\ i\neq k}}^p 
	\lb \E [ Y_{kj}^2 Y_{ij}^2] - n^{-2} \rb 
	= o(1)  
	\end{align*}
	follows from \eqref{est_wn2} in the proof of Lemma \ref{lem_quad_form}. 
	\end{proof}

	\begin{proposition} \label{thm_first_mom}
	For all $1\leq j \leq n$, we have
		\begin{align*}
			\max\limits_{1 \leq k \leq p} | \E [  Y_{k1} ] |  = \max\limits_{1 \leq k \leq p}  | \E [ Y_{kj} ] | = o(n\inv)\,, \qquad \nto\,.
	\end{align*}
	\end{proposition}
	
	\begin{proof}
	To begin with, we truncate the random variable $X_{k1}$ using Lemma \ref{lem_pr_delta}. 
	Note that 
	\begin{align*}
		n \left| \E [  Y_{k1} ] -  \E [ Y_{k1} \1\{ | X_{k1}| \leq \sqrt{n} \delta_n \} ] \right| 
		= n \left|  \E [ Y_{k1} \1\{ | X_{k1}| > \sqrt{n} \delta_n \} ] \right| 
		\leq n \PR \lb  | X_{k1}| > \sqrt{n} \delta_n \rb = o(1)
	\end{align*}
	uniformly over $1 \leq k \leq p$. 
	Combining \eqref{formula_inv} with Fubini's theorem, we deduce
		\begin{align}
			  \E&[  Y_{k1} \1\{ | X_{k1}| \leq \sqrt{n} \delta_n \} ]  
			 =   \E \left[ \frac{1}{\Gamma \lb \frac{1}{2} \rb }  
			\int\limits_0^\infty X_{k1}  \exp\lb - s \sum\limits_{j=1}^n X_{kj}^2 \rb  s^{-\frac{1}{2}}ds \, \1\{ | X_{k1}| \leq \sqrt{n} \delta_n \} \right] \nonumber
		\\
			&=   \frac{1}{\Gamma \lb \frac{1}{2} \rb }
			\int\limits_0^\infty \E \left[  X_{k1}  \exp\lb - s  X_{k1}^2 \rb \1\{ | X_{k1}| \leq \sqrt{n} \delta_n \} \right] 
			\E \left[ \exp\lb - s \sum\limits_{j=2}^n X_{kj}^2 \rb \right] s^{-\frac{1}{2}} ds \nonumber  \\
			&=  \frac{1}{\Gamma \lb \frac{1}{2} \rb }
			\int\limits_0^\infty \E \left[   X_{k1}   \exp\lb - s  X_{k1}^2 \rb \1\{ | X_{k1}| \leq \sqrt{n} \delta_n \} \right]
			 \lb \varphi_{k}(s) \rb^{n-1} s^{-\frac{1}{2}} ds , \nonumber
		\end{align}
		where
		\begin{align}
			\varphi_{k} (s) = \E \left[ \exp\lb - s  X_{k1}^2 \rb \right], ~ s>0, 
			\label{def_laplace}
		\end{align}
		denotes the Laplace transform of $X_{k1}^2$, $1\leq k \leq p$. 
	Let $\varepsilon>0$. 
	Lemma \ref{lem_int} implies that
	\begin{align}
		n  | \E &[ Y_{k1}  \1\{ | X_{k1}| \leq \sqrt{n} \delta_n \}] | \nonumber\\
		&\leq 
		 \frac{n}{\Gamma \lb \frac{1}{2} \rb }
			\int\limits_{\varepsilon}^\infty \E \left[ |  X_{k1} |  \exp\lb - s  X_{k1}^2 \rb \1\{ | X_{k1}| \leq \sqrt{n} \delta_n \} \right] \lb \varphi_{k}(s) \rb^{n-1} s^{-\frac{1}{2}} ds
		+ o(1) \nonumber\\ 
		&\leq  \frac{n}{\Gamma \lb \frac{1}{2} \rb }
			\int\limits_{\varepsilon}^\infty \E \left[ |  X_{k1} |  \exp\lb - s  X_{k1}^2 \rb  \right] \lb \varphi_{k}(s) \rb^{n-1} s^{-\frac{1}{2}} ds
		+ o(1), \label{int_first_mom}
	\end{align}
	where the symbol $o(1)$ holds uniformly in $k\in\N$. 
	Invoking Lemma \ref{lem_laplace_bound} and Lemma \ref{lem_deriv_laplace}, we conclude that 
	\begin{align*}
		\lim\limits_{n \to \infty} \max\limits_{1 \leq k \leq p} \frac{n \varphi^n_k(\varepsilon) }{\varphi_k(\varepsilon) } 
		= 0.
	\end{align*}
	Combining this observation with the estimate 
	\begin{align*}
		 n
			\int\limits_{\varepsilon}^\infty \E \left[ |  X_{k1} |  \exp\lb - s  X_{k1}^2 \rb \right] \lb \varphi_{k}(s) \rb^{n-1} s^{-\frac{1}{2}} ds
			\leq 
			\frac{ n \varphi_k^n(\varepsilon)}{\varphi_k (\varepsilon)} 
			\int_\varepsilon^\infty \E \left[ |  X_{k1} |  \exp\lb - s  X_{k1}^2 \rb \right] s^{-\frac{1}{2}} ds,
	\end{align*}
	the assertion finally follows, since
	\begin{align*}
		&\int\limits_\varepsilon^\infty \E \left[ |  X_{k1} |  \exp\lb - s  X_{k1}^2 \rb \right] s^{-\frac{1}{2}} ds 
		\leq  \int\limits_0^\infty \E \left[ |  X_{k1} |  \exp\lb - s  X_{k1}^2 \rb \right] s^{-\frac{1}{2}} ds \\
		&=   \E \left[ \int\limits_0^\infty 
		 |  X_{k1} |  \exp\lb - s  X_{k1}^2 \rb  s^{-\frac{1}{2}} ds \right] 
		 =  \E \left[ \int\limits_0^\infty 
		 \exp\lb - s   \rb  s^{-\frac{1}{2}} ds \right] 
		 = \Gamma \lb \frac{1}{2} \rb. 
	\end{align*}
\end{proof}

%
	\begin{proposition} \label{thm_4mom}
		It holds
		\begin{align*}
			\lim\limits_{n\to\infty} \max\limits_{1 \leq k \leq p} n \E [ Y_{k1}^4] = 0. 
		\end{align*}
	\end{proposition}
	\begin{proof}
		Using \eqref{formula_inv} and Fubini's theorem, we obtain
		\begin{align*}
			n \E [ Y_{k1}^4] = n \int\limits_0^\infty s \varphi_k''(s) \varphi_k^{n-1} (s) ds ,
	\end{align*}			
	where $\varphi_k$ denotes the Laplace transform of $X_{k1}^2$ defined in \eqref{def_laplace}. 
	Let $\varepsilon >0$.
	By Lemma \ref{lem_cn}, it suffices to show that
	\begin{align} \label{conv_an}
	\lim\limits_{n\to\infty} \max\limits_{1 \leq k \leq p} 
		n \int\limits_0^\varepsilon s \varphi_k''(s) \varphi_k^{n-1} (s) ds 
		\leq  \lim\limits_{n\to\infty} 
		n \int\limits_0^\varepsilon s \max\limits_{1 \leq k \leq p}  \lb  \varphi_k''(s) \varphi_k^{n-1} (s) \rb ds 
		= 0 . 
	\end{align}
	In order to apply the dominated convergence theorem, we will first argue that the integrand
	\begin{align} \label{integrand}
		\max\limits_{1 \leq k \leq p} n s \varphi_k''(s) \varphi_k^{n-1} (s), ~ \quad 0 < s <\varepsilon,
	\end{align}
	is dominated by an integrable function independent of $n\in\N$. 
	For this purpose, note that 
	\begin{align*}
	 \varphi_k''(s) 
		= \E \left[ X_{k1}^{2+\delta} \lb X_{k1}^{2-\delta} \exp\lb - s X_{k1} ^2 \rb \rb \right] \lesssim \E \left[ |X_{k1}|^{2+\delta}  \right] s^{-1+0.5\delta} \lesssim s^{-1+0.5\delta}, ~ 1 \leq k \leq p,
	\end{align*}
	where the latter is integrable on $(0, \varepsilon)$. Here, we used that $\E [ \xi^{2+\delta} ] < \infty$ and that the function $f(x) = x^{2-\delta} \exp( - s x^2 )$ has extremal points at $x= \pm \sqrt{( 2 - \delta ) / 2s }$.  
	Moreover, we obtain 
	\begin{equation} \label{est_varphi}
		\max_{1 \leq k \leq p}   n s \varphi_k^{n-1} (s) 
		 \leq \max\limits_{1 \leq k \leq p}  \frac{ e\inv }{\varphi_k(\varepsilon) ( - \varphi_k''(\varepsilon) )  } 
		 \lesssim 1 , \quad ~ 0 < s < \varepsilon, 
	\end{equation}
	where we proceeded similarly as in the proof of Theorem 3.2 in \cite{fuchs:joffe:teugels:2001} for the first inequality and for the second one, we used Lemma \ref{lem_deriv_laplace}. \\ 
	Finally, we observe that the integrand \eqref{integrand} converges to zero. More precisely, we obtain for $0 < s < \varepsilon$, 
	 \begin{align*}
	 	\eqref{integrand} 
	 	\leq \max\limits_{1 \leq k \leq p}   n \varepsilon \varphi_k''(s) \varphi_k^{n-1} (s)
	 	\lesssim \max\limits_{1 \leq k \leq p}   n  \varphi_k^{n-1} (s)
	 	= o(1),
	 \end{align*}
	 where we used Lemma \ref{lem_laplace_bound}.
	Summarizing, an application of the dominated convergence theorem implies \eqref{conv_an} and thus finishes the proof of Proposition \ref{thm_4mom}. 
	\end{proof}

	
	\appendix

	 \section{Useful results}
	    By Theorem 2.7 in \cite{yao:zheng:bai:2015}, we have the following lemma. 
\begin{lemma}\label{lem:stieltjes} 
Let $\mu, \mu_1, \mu_2, \ldots$ be (random) probability measures with support in $\R^+$. Then $\mu_n$ converges weakly to $\mu$ almost surely if and only if $s_{\mu_n}(z) \to s_{\mu}(z)$ almost surely for all $z\in\mathbb{C}^+$.
\end{lemma}

\begin{lemma}\label{lemma3.4} 
Let $(Z_n)_{n\in\N}$ be complex valued random variables such that $\sup_{n\in\N} |Z_n|$ is bounded almost surely. If $(Y_n)_{n\in \N}$ are random variables satisfying 
\begin{equation*}
\lim_{\nto} \E\Big[ \frac{Z_n}{1+Z_n}\Big] -\E\Big[ \frac{\E[Z_n|Y_n]}{1+\E[Z_n|Y_n]}\Big]=0\,,
\end{equation*}
then $Z_n-\E[Z_n|Y_n] \cip 0$, as $\nto$. 
\end{lemma}
\begin{proof}
The proof is very similar to the proof of Lemma 3.4 in \cite{yaskov:2016} and, thus, is omitted for the sake of brevity. 
\end{proof}

We conclude this section by collecting some useful inequalities for matrices; see, e.g., \cite{bai:silverstein:2010}. 
 \begin{lemma}
 For a real, symmetric, positive semidefinite $p\times p$ matrix $\bfC$, $x\in \R^p$, $z\in \C^+$ with $\Im(z)=v>0$ the following inequalities hold:
\begin{equation}\label{eq:le1}
\twonorm{(\bfC-z\bfI)^{-1}}\le \tfrac{1}{v}\,,
\end{equation}
\begin{equation}\label{eq:le4}
\Im\big( z+z \tr\big( (\bfC-z\bfI)^{-1} \big) \big)  \ge v \quad \text{ and }  \quad
\Im\big(\tr\big( (\bfC-z\bfI)^{-1} \big) \big)>0\,,
\end{equation}
\begin{equation}\label{eq:le5}
\Im\big( z+z x^\top(\bfC-z\bfI)^{-1} x \big)  \ge v\,.
\end{equation}
\end{lemma}
    \section{Proof of Lemma \ref{lem_quad_form} in Section \ref{sec_quad_forms}} \label{appendix_c}
    In this section, we give a proof for Lemma \ref{lem_quad_form} which is one of the main ingredients for proving the main result for the dependent case. 
    Consequently, throughout this section, we work under the assumptions of Section \ref{sec_dependent}. 
	\begin{proof}[Proof of Lemma \ref{lem_quad_form}]
	For convenience, we suppress the dependency on $j\in\{1, \ldots, n\}$ of the matrix $\mathbf{B}^{(j)} = \mathbf{B}$ by our notation, that is, we denote its entries by $B_{ik}$ instead of $B_{ik}^{(j)}$, $1 \leq i,k \leq p$. 
	We have (using that $\diag(\bfT) = \bfI$ and $\E [Y_{kj}^2] = n\inv$)
	\begin{align}
		V_n  = &  \sum\limits_{k,l,m,r=1}^p B_{kl} \overline{B}_{mr} \E [ Y_{kj} Y_{lj} Y_{mj} Y_{rj} ] 
		-  n\invv \sum\limits_{l,k=1}^p \E [Y_{lj} Y_{kj}] \overline{B}_{kl} \sum\limits_{i,m=1}^p T_{im} B_{mi} \nonumber \\
		& -  n\invv \sum\limits_{l,k=1}^p \E [Y_{lj} Y_{kj}] B_{kl} \sum\limits_{i,m=1}^p T_{im} \overline{B}_{mi} 
		+ n^{-2} \sum\limits_{i,l,r,k=1}^p T_{il} B_{li} T_{kr} \overline{B}_{rk} \nonumber \\
		= & \sum\limits_{q=1}^7 V_{n,q} + o(1), \label{rep_wn}
	\end{align}
		where
	\begin{align*} 
	V_{n,1} = & \sum\limits_{k=1}^p  | B_{kk} |^2 \lb \E [ Y_{kj}^4] - n^{-1} T_{kk} \E [Y_{kj}^2] \rb 
	= \sum\limits_{k=1}^p | B_{kk} |^2 \lb \E [ Y_{kj}^4] - n^{-2} \rb, \\
	V_{n,2} = & \sum\limits_{\substack{k,l=1, \\ k\neq l}}^p B_{kk} \overline{B}_{ll} \lb \E [ Y_{kj}^2 Y_{lj}^2] - n\invv T_{ll} \E [Y_{kj}^2] \rb
	=   \sum\limits_{\substack{k,l=1, \\ k\neq l}}^p B_{kk} \overline{B}_{ll} \lb \E [ Y_{kj}^2 Y_{lj}^2] - n^{-2} \rb ,\\
	V_{n,3} = & \sum\limits_{\substack{k,l=1, \\ k\neq l}}^p \lb | B_{kl} |^2 + B_{kl} \overline{B}_{lk} \rb 
	\E [ Y_{kj}^2 Y_{lj}^2] ,\\
	 V_{n,4} = & n\inv \tr ( \overline{\mathbf{B}} \T ) 
	\sum\limits_{\substack{l,k=1, \\ k \neq l}}^p B_{kl} \lb n\inv T_{lk} - \E [ Y_{lj} Y_{kj} ] \rb, \\
	V_{n,5} = & - n\inv \sum\limits_{\substack{l,k=1, \\ l \neq k }}^p \E [ Y_{kj} Y_{lj} ] \overline{B}_{kl} \tr ( \mathbf{B} \T ), \\
	V_{n,6} = & - n^{-2} \tr ( \overline{\mathbf{B}} ) \sum\limits_{\substack{m,i=1, \\ m\neq i} }^p T_{im} B_{mi}, \\
	V_{n,7} = & \sum\limits_{\substack{k,l,m,i=1, \\ | \{k,l,m,i\}| \geq 3 } }^p B_{kl} \overline{B}_{mi} \E [ Y_{kj} Y_{lj} Y_{mj} Y_{ij} ] .
	\end{align*}	
	For the estimate in \eqref{rep_wn}, we used that by \eqref{first_mom}, \eqref{4th_mom} and assumption \ref{a_sparse_U} we have uniformly in $j$
	\begin{align*}
		\sum\limits_{\substack{k,l=1, \\ |\{k,l\} |=2 }}^p  \!\!\!\! \left| \E [ Y_{kj}^3 Y_{lj}   ] \right|
		= \!\!\!\! \sum\limits_{\substack{k,l=1, \\ |\{k,l\} |=2, \\ k \in \mathcal{I}(l)  }}^p  \!\!\!\! \left| \E [ Y_{kj}^3 Y_{lj}   ] \right| 
		+ \!\!\!\! \sum\limits_{\substack{k,l=1, \\ |\{k,l\} |=2, \\ k \notin \mathcal{I}(l)  }}^p  \!\!\!\! \left| \E [ Y_{kj}^3 ] \E [ Y_{lj}   ] \right| 
		\leq \!\!\!\! \sum\limits_{\substack{k,l=1, \\ |\{k,l\} |=2, \\ k \in \mathcal{I}(l)  }}^p \!\!\!\! \lb \E [ Y_{kj}^4] \E [ Y_{lj}^2   ] \rb \sq + o(1)  =o(1).
	\end{align*}

We aim to show that 
	\begin{align} \label{aim_w}
		\sum\limits_{q=1}^7 V_{n,q} = o(1), ~ \quad n\to\infty.
\end{align}	 
	In order to prove \eqref{aim_w}, we first note that $V_{n,1} = o(1)$ due to \eqref{4th_mom}.
	For the second summand $V_{n,2}$, we estimate
	\begin{align}
		| V_{n,2} | \lesssim &  \sum\limits_{\substack{k,l=1, \\ k\neq l}}^p 
		\left| \E [ Y_{kj}^2 Y_{lj}^2] - n^{-2} \right|
		= \sum\limits_{\substack{k,l=1, \\ k\neq l, \\ \mathcal{I}(k) \cap \mathcal{I}(l) = \emptyset }}^p 
		\left| \E [ Y_{kj}^2 ] \E [ Y_{lj}^2] - n^{-2} \right|
		+ \sum\limits_{\substack{k,l=1, \\ k\neq l, \\ \mathcal{I}(k) \cap \mathcal{I}(l) \neq \emptyset }}^p 
		\left| \E [ Y_{kj}^2 Y_{lj}^2] - n^{-2} \right| \nonumber \\
		= &  \sum\limits_{\substack{k,l=1, \\ k\neq l, \\ \mathcal{I}(k) \cap \mathcal{I}(l) \neq \emptyset }}^p 
		\left| \E [ Y_{kj}^2 Y_{lj}^2] - n^{-2} \right| = o(1), \label{est_wn2}
	\end{align}
	where we used \eqref{4th_mom} and $|\mathcal{I}(k) \cap \mathcal{I}(l) | \lesssim 1$. 
	Regarding $V_{n,3}$, we obtain similarly 
	\begin{align*}
		V_{n,3} =
		 \frac{1}{n^2} \sum\limits_{\substack{k,l=1, \\ k\neq l, \\ \mathcal{I}(k) \cap \mathcal{I}(l) = \emptyset }}^p \lb | B_{kl} |^2 + B_{kl} \overline{B}_{lk} \rb 
	+ \sum\limits_{\substack{k,l=1, \\ k\neq l, \\ \mathcal{I}(k) \cap \mathcal{I}(l) \neq \emptyset }}^p \lb | B_{kl} |^2 + B_{kl} \overline{B}_{lk} \rb 
	\E [ Y_{kj}^2 Y_{lj}^2] 
	= : V_{n,3,1} + V_{n,3,2},
	\end{align*}
	where, with $\star$ denoting the conjugate transpose of a matrix,
	\begin{align*}
	| V_{n,3,1} | \leq & n^{-2} \left[ \tr \lb  \mathbf{B}^{(j)} \lb \mathbf{B}^{(j)}\rb^\star \rb + \tr  \left|  \mathbf{B}^{(j)} \overline{\mathbf{B}^{(j)}}  \right| \right]  = o(1), \\
	| V_{n,3,2}| \leq & \max\limits_{1 \leq k \leq p} \E [ Y_{k1}^4]  \sum\limits_{\substack{k,l=1, \\ k\neq l, \\ \mathcal{I}(k) \cap \mathcal{I}(l) \neq \emptyset }}^p \lb | B_{kl} |^2 +|  B_{kl} \overline{B}_{lk} | \rb = o(1). 
	\end{align*}
	Next, we obtain for $V_{n,4}$ 
	\begin{align*}
		| V_{n,4} | \lesssim & \sum\limits_{\substack{k,l=1, \\ k\neq l}}^p 
		\left| \E [ Y_{kj} Y_{lj} ] - n\inv T_{lk} \right| 
		= \sum\limits_{\substack{k,l=1, \\ k\neq l, \\ \mathcal{I}(l) \cap \mathcal{I}(k) = \emptyset}}^p 
		\left| \E [ Y_{kj} ] \E [ Y_{lj} ]  \right| 
		+ \sum\limits_{\substack{k,l=1, \\ k\neq l, \\ \mathcal{I}(l) \cap \mathcal{I}(k) \neq \emptyset}}^p 
		\left| \E [ Y_{kj} Y_{lj} ] - n\inv T_{lk} \right| \\
		= & o(1)+\sum\limits_{\substack{k,l=1, \\ k\neq l, \\ \mathcal{I}(l) \cap \mathcal{I}(k) \neq \emptyset}}^p 
		\left| \E [ Y_{kj} Y_{lj} ] - n\inv T_{lk} \right| = o(1) , \\
	\end{align*}
	where we used \eqref{first_mom}
	and the fact that $T_{lk} = 0$ follows from $ \mathcal{I}(l) \cap \mathcal{I}(k) = \emptyset $.
	We also used \ref{a_sparse_U} combined with $\left| \E [ Y_{kj} Y_{lj} ] - n\inv T_{lk} \right| = o(n\inv)$, which follows from formula (4) in \cite{lai:rayner:hutchinson:1999}. \\ 
	Investigating $V_{n,5}$ further, we write
		$V_{n,5} = V_{n,5,1} + V_{n,5,2} + V_{n,5,3}$, 
	where
	\begin{align*}
		V_{n,5,1} = & - n\inv \sum\limits_{\substack{ k,l,m,i =1, \\ | \{ k,l,m,i\} | = 4 }}^p \overline{B}_{kl} B_{mi} T_{im} \E [ Y_{kj} Y_{lj} ] , \\
		V_{n,5,2} = & - n\inv \sum\limits_{\substack{k,l,m=1, \\ |\{ k,l,m\} | = 3 }} 
		\overline{B}_{kl} B_{mm} \E [ Y_{kj} Y_{lj} ] , \\
		V_{n,5,3} = & - n\inv \sum\limits_{\substack{k,l=1, \\ k\neq l}}^p \lb \overline{B}_{kl} \E [ Y_{kj} Y_{lj} ] \sum\limits_{\substack{m,i=1,~ m\neq i \\ m\in\{l,k\} \textnormal{ or } i \in \{l,k\} }}^p T_{im} B_{mi} \rb . 
	\end{align*}
	Note that due to \eqref{first_mom} 
	\begin{align*}
		- V_{n,5,3} 
		&=  n\inv \sum\limits_{\substack{k,l=1, \\ k\neq l, \\ l \notin \mathcal{I}(k) }}^p \lb \overline{B}_{kl} \E [ Y_{kj} ] \E [ Y_{lj} ] \sum\limits_{\substack{m,i=1,~ m\neq i \\ m\in\{l,k\} \textnormal{ or } i \in \{l,k\} }}^p T_{im} B_{mi} \rb \\
		&\quad + n\inv \sum\limits_{\substack{k,l=1, \\ k\neq l, \\ l \in \mathcal{I}(k) }}^p \lb \overline{B}_{kl} \E [ Y_{kj} Y_{lj} ] \sum\limits_{\substack{m,i=1,~ m\neq i \\ m\in\{l,k\} \textnormal{ or } i \in \{l,k\} }}^p T_{im} B_{mi} \rb \\
		&=  n\inv \sum\limits_{\substack{k,l=1, \\ k\neq l, \\ l \in \mathcal{I}(k) }}^p \lb \overline{B}_{kl} \E [ Y_{kj} Y_{lj} ] \sum\limits_{\substack{m,i=1,~ m\neq i \\ m\in\{l,k\} \textnormal{ or } i \in \{l,k\}, \\ i \in \mathcal{I}(m) }}^p T_{im} B_{mi} \rb + o(1) = o(1),
	\end{align*}
	where we used \ref{a_sparse_U} and $| \E [ Y_{kj} Y_{lj} ] | \leq n\inv $ by \Holder's inequality. 
	Combining $V_{n,5,1}$ with the corresponding summand in $V_{n,7}$ ($|\{k,l,m,i\} | =4$), we have
	\begin{align*}
	&V_{n,5,1} + \sum\limits_{\substack{k,l,m,i=1, \\ | \{k,l,m,i\}| = 4 } }^p B_{kl} \overline{B}_{mi} \E [ Y_{kj} Y_{lj} Y_{mj} Y_{ij} ] 
	=  \sum\limits_{\substack{k,l,m,i=1, \\ | \{k,l,m,i\}| = 4 } }^p 
	B_{kl} \overline{B}_{mi} \lb \E [ Y_{kj} Y_{lj} Y_{mj} Y_{ij} ]  - n\inv T_{im} \E [ Y_{kj} Y_{lj} ] \rb \\
	&=  \sum\limits_{\substack{k,l,m,i=1, \\ | \{k,l,m,i\}| = 4 } }^p 
	B_{kl} \overline{B}_{mi} \lb \E [ Y_{kj} Y_{lj} Y_{mj} Y_{ij} ]  - \E [ Y_{ij} Y_{mj}] \E [ Y_{kj} Y_{lj} ] \rb + o(1) \\
	&=  \sum\limits_{\substack{k,l,m,i=1, \\ | \{k,l,m,i\}| = 4 , \\ \lb \mathcal{I}(k) \cup \mathcal{I}(l) \rb \cap \lb \mathcal{I}(m) \cup \mathcal{I}(i) \rb \neq \emptyset } }^p 
	B_{kl} \overline{B}_{mi} \cov (Y_{kj} Y_{lj}, Y_{mj} Y_{ij} ) + o(1) =  o(1)\,. 
	\end{align*}
	 For the last equality, we used that, if one random variable, say $Y_{kj}$, is independent of $Y_{ij}, Y_{kj}$ and $ Y_{lj}$, then the corresponding covariance term  satisfies due to \eqref{first_mom}
	 \begin{align*}
	 	| \cov (Y_{kj} Y_{lj}, Y_{mj} Y_{ij} ) | 
	 	= \left| \E [Y_{kj}] \E[ Y_{ij} Y_{kj} Y_{lj}] - \E [ Y_{kj}] \E[Y_{lj}] \E [ Y_{mj} Y_{ij} ] \right|
	 	= o \lb n^{-2} \rb,
	 \end{align*} 
	 and in this case, we have $\mathcal{O}(n^2)$ summands. 
	 Otherwise, we use the estimate
	 \begin{align*}
	 	\cov (Y_{kj} Y_{lj}, Y_{mj} Y_{ij} ) = o\lb n\inv \rb
	 \end{align*}
	 and note that we only have $\mathcal{O}(n)$ summands in this case due to \ref{a_sparse_U}. 
	 	  Next, we combine $V_{n,5,2}$ with a corresponding summand in $V_{n,7}$ ($k=l, ~ |\{ k,m,i\} | =3 $) and get
	  \begin{align}
	  	& V_{n,5,2} + \sum\limits_{\substack{k,m,i=1, \\ | \{k,m,i\}| =3 } }^p B_{kk} \overline{B}_{mi} \E [ Y_{kj}^2 Y_{mj} Y_{ij} ]
	  	=  \sum\limits_{\substack{k,m,i=1, \\ | \{k,m,i\}| =3 } }^p B_{kk} \overline{B}_{mi} \lb  \E [ Y_{kj}^2 Y_{mj} Y_{ij} ] - n\inv \E [ Y_{mj} Y_{ij} ] \rb \nonumber \\
	  	= & \sum\limits_{\substack{k,m,i=1, \\ | \{k,m,i\}| =3, \\ k \in \mathcal{I}(m) \cup \mathcal{I}(i) } }^p B_{kk} \overline{B}_{mi}
	  	\cov ( Y_{kj}^2, Y_{mj} Y_{ij} )  \label{b1} \\
	  	= & \sum\limits_{\substack{k,m,i=1, \\ | \{k,m,i\}| =3, \\ k \in \mathcal{I}(m) \cup \mathcal{I}(i), \\ \mathcal{I}(m) \cap \mathcal{I}(i) \neq \emptyset } }^p B_{kk} \overline{B}_{mi}
	  	\cov ( Y_{kj}^2, Y_{mj} Y_{ij} )  
	  	 + \sum\limits_{\substack{k,m,i=1, \\ | \{k,m,i\}| =3, \\ k \in \mathcal{I}(m) \cup \mathcal{I}(i), \\  \mathcal{I}(m) \cap \mathcal{I}(i) = \emptyset } }^p B_{kk} \overline{B}_{mi}
	  	\cov ( Y_{kj}^2, Y_{mj} Y_{ij} )
	  	= o(1), \nonumber 
	  \end{align}
	  where we used 
	  \begin{align*}
	  	| \cov ( Y_{kj}^2, Y_{mj} Y_{ij} ) | 
	  	= \begin{cases}
	  	 o \lb n\inv \rb & \textnormal{ if }  k \in \mathcal{I}(m) \cup \mathcal{I}(i) \textnormal{ and } \mathcal{I}(m) \cap \mathcal{I}(i) \neq \emptyset, \\
	  	o \lb n^{-2} \rb & \textnormal{ if } k \in \mathcal{I}(m) \cup \mathcal{I}(i) \textnormal{ and }  \mathcal{I}(m) \cap \mathcal{I}(i) = \emptyset,
	  	 \end{cases} 
	  \end{align*}
	  and \eqref{first_mom}, \eqref{4th_mom} as well as assumption \ref{a_sparse_U}.
	
	  Considering $V_{n,6}$, we decompose
	  \begin{align*}
	  	- V_{n,6} = & n^{-2} \sum\limits_{\substack{i,k,m=1, \\ |\{ i,k,m\} | = 3 }}^p \overline{B}_{kk} T_{im} B_{mi} 
	  	+ n^{-2} \sum\limits_{\substack{i,k,m=1, \\ i \neq m , \\ k = i \textnormal{ or } k =m }}^p \overline{B}_{kk} T_{im} B_{mi} \\
	  	= &  n^{-2} \sum\limits_{\substack{i,k,m=1, \\ |\{ i,k,m\} | = 3 }}^p \overline{B}_{kk} T_{im} B_{mi} 
	  	+ n^{-2} \sum\limits_{\substack{i,k,m=1, \\ i \neq m , \\ k = i \textnormal{ or } k =m, \\ i \in \mathcal{I}(m) }}^p \overline{B}_{kk} T_{im} B_{mi} \\
	  	= &  n^{-2} \sum\limits_{\substack{i,k,m=1, \\ |\{ i,k,m\} | = 3 }}^p \overline{B}_{kk} T_{im} B_{mi}  + o(1),
	  \end{align*}
	  where we used \ref{a_sparse_U}. Combining this term with the corresponding term in $V_{n,7}$ ($m=i, ~|\{k,l,m\}|=3$) we have
	  \begin{align*}
	  	V_{n,6} + \sum\limits_{\substack{k,l,m=1, \\ | \{k,l,m\}| = 3 } }^p B_{kl} \overline{B}_{mm} \E [ Y_{kj} Y_{lj} Y_{mj}^2 ] 
	  	= &  \sum\limits_{\substack{k,l,m=1, \\ | \{k,l,m\}| = 3 } }^p B_{kl} \overline{B}_{mm} \lb \E [ Y_{kj} Y_{lj} Y_{mj}^2 ] - n^{-2} T_{kl} \rb 
	  	+o(1) \\
	  	= &   \sum\limits_{\substack{k,l,m=1, \\ | \{k,l,m\}| = 3 } }^p B_{kl} \overline{B}_{mm} \lb \E [ Y_{kj} Y_{lj} Y_{mj}^2 ] - n^{-1} \E [ Y_{kj} Y_{lj} ]  \rb 
	  	+o(1)  \\ 
	  	= & \sum\limits_{\substack{k,l,m=1, \\ | \{k,l,m\}| = 3, \\ m\in \mathcal{I}(k) \cup \mathcal{I}(l) } }^p B_{kl} \overline{B}_{mm} \cov( Y_{kj} Y_{lj}, Y_{mj}^2 ) 
	  	+o(1) 
	  	= o(1),
	\end{align*}	   
	where we concluded similarly to \eqref{b1} for the last estimate. 
	Finally, we devote our attention to the remaining terms in $V_{n,7}$, which are
	\begin{align*}
	& \sum\limits_{ \substack{ k,m,i=1, \\ |\{k,m,i\}| = 3}} 
	B_{km} \overline{B_{mi}} \E [ Y_{kj} Y_{ij} Y_{mj}^2 ], 
	\qquad \sum\limits_{ \substack{ k,m,i=1, \\ |\{k,m,i\}| = 3}} 
	B_{kl} \overline{B_{mk}} \E [ Y_{kj}^2 Y_{lj} Y_{mj} ], \\
	& \sum\limits_{ \substack{ k,m,i=1, \\ |\{k,m,i\}| = 3}} 
	B_{kl} \overline{B_{ki}} \E [ Y_{kj}^2 Y_{lj} Y_{ij} ], 
	\qquad \sum\limits_{ \substack{ k,m,i=1, \\ |\{k,m,i\}| = 3}} 
	B_{kl} \overline{B_{ml}} \E [ Y_{kj} Y_{lj}^2 Y_{mj} ].
	\end{align*}
	Exemplarily, we consider
	\begin{align*}
		 & \left| \sum\limits_{ \substack{ k,m,i=1, \\ |\{k,m,i\}| = 3}}^p
	 B_{km} \overline{B_{mi}} \E [ Y_{kj} Y_{ij} Y_{mj}^2 ] \right| 
	 \leq \sum\limits_{ \substack{ k,m,i=1, \\ |\{k,m,i\}| = 3}}^p
	 \left| B_{km} \overline{B_{mi}} \E [ Y_{kj} Y_{ij} Y_{mj}^2 ] \right| \\
	 = & o(n\inv) \sum\limits_{\substack{k,i=1, \\ k \neq i }}^p
	\sum\limits_{\substack{m=1, \\ m\notin \{k,i\} }}^p \left| B_{km} B_{mi} \right| 
	=  o(1).
	\end{align*}
	The other terms can be shown to be asymptotically negligible in a similar way. 
Thus, \eqref{aim_w} holds true and since our estimates did not depend on $j$,  \eqref{aim_w} holds also uniformly in $j$.
	\end{proof}

	\section{Properties of the Laplace Transform $\varphi_k$} \label{appendix_laplace}
	In the following, we investigate the Laplace transform $\varphi_k$ of $X_{k1}^2$ further and provide estimates for integrals involving this function. 
	Throughout this section, we work under the assumptions of Section~\ref{sec_dependent} if not explicitly stated otherwise.  
	\begin{lemma} \label{lem_pr_delta}
	There exists a positive sequence $(\delta_n)_{n\in\N}$ independent of $1\leq k \leq p$ converging to zero and satisfying
	\begin{align*} 
		\lim\limits_{n\to\infty} n \max\limits_{1 \leq k \leq p} \PR ( | X_{k1} | > \sqrt{n} \delta_n) = 0.
	\end{align*}
	\end{lemma}
	
	\begin{proof}
	Since
	$1 = \E [ \tilde{X}_{11}^2 ] = 2 \int_0^\infty x \PR (|\tilde{X}_{11} | > x) dx$,
	there exists a positive sequence $(\tilde{\delta}_n)_{n\in\N}$ converging to zero with the property
	\begin{align*}
		\lim\limits_{n\to\infty} n \PR ( | \tilde{X}_{11} | > \sqrt{n} \tilde{\delta}_n ) = 0.
	\end{align*}
	Moreover, we observe for any $y>0$
	\begin{align*}
		\{ | X_{k1} | > y \} 
		\subset \left\{ \sum\limits_{l \in \mathcal{I}(k) } | U_{kl} \tilde{X}_{l1} | > y \right\} \subset \left\{ \sum\limits_{l \in \mathcal{I}(k) } | \tilde{X}_{l1} | > y \right\} 
		\subset \left\{ \max\limits_{l \in \mathcal{I}(k) } |  \tilde{X}_{l1} | > \frac{y}{q} \right\},
	\end{align*}		
	where, by assumption \ref{a_sparse_U},
		$q := \sup_{n\in\N} \max_{1 \leq k \leq p}  | \mathcal{I}(k) | < \infty$. 
	Using
	\begin{align*}
		\max\limits_{1 \leq k \leq p} \PR ( | X_{k1} | > y ) 
		\leq \max\limits_{1 \leq k \leq p} \PR \lb \max\limits_{l \in \mathcal{I}(k) } |  \tilde{X}_{l1} | > \frac{y}{q} \rb
		\leq q \PR \lb | \tilde{X}_{11} | > \frac{y}{q} \rb ,
	\end{align*}		
	we conclude that the assertion of Lemma \ref{lem_pr_delta} holds true if we set $\delta_n = \tilde{\delta}_n / q$ for $n\in\N$. 
	\end{proof}

	The following two lemmas are generalizations of Lemma 3.1 and Lemma 3.5 in \cite{fuchs:joffe:teugels:2001}, respectively. 
	\begin{lemma} \label{lem_laplace_bound}
	For every $s>0$, the Laplace transforms $\varphi_k(s)$ are uniformly bounded by one, that is,
\begin{align*}
			\sup_{n \in \N}\max_{1\le k\le p} \varphi_k(s) < 1. 
\end{align*}	
	\end{lemma} 	
	\begin{proof}
	Assume that there exists some $\varepsilon > 0$ with the property 
	\begin{align} \label{aim_eps}
	    \sup_{n\in\N} \max_{1 \leq k \leq p} \PR ( X_{k1}^2 \leq \varepsilon ) <1. 
	\end{align}
	Then, it follows for every $s>0$
	\begin{align*}
	    \varphi_k (s) = & \E \left[ \exp(-s X_{k1}^2 ) I \{ X_{k1}^2 \leq \varepsilon \} \right] 
	    + \E \left[ \exp(-s X_{k1}^2 ) I \{ X_{k1}^2 > \varepsilon \} \right] \\
	    \leq & \PR ( X_{k1}^2 \leq \varepsilon ) + \exp(-s\varepsilon) \PR ( X_{k1}^2 > \varepsilon ) 
	    = \lb 1 - \exp(-s \varepsilon) \rb  \PR \lb X_{k1}^2 \leq \varepsilon \rb  + \exp(-s\varepsilon)  ,
	\end{align*}
	which implies the assertion of Lemma \ref{lem_laplace_bound}. 
	Thus, it is left to show that \eqref{aim_eps} holds true, which will be proven by contradiction. Assume that \eqref{aim_eps} does not hold. 
	Then, one can obtain a sequence $(X_{k(n)1}^2)_{n\in\N}$ which converges in probability to zero as $n\to\infty$.
	Using $\E | X_{k1} |^{2+\delta} < \infty$ uniformly over $1 \leq k \leq p, ~ n\in\N$ by assumption \ref{a_2mom}, we note that
	\begin{align*}
	    \sup_{n\in\N} \E \left[ X_{k(n)1}^2 \1\{ X_{k(n)1}^2 > c \} \right]
	    \leq \frac{1}{c^\delta} \sup\limits_{n\in\N} \E \left[ X_{k(n)1}^{2 + \delta} \right] 
	    \to 0, ~ c \to\infty, 
	\end{align*}
	which shows that the sequence $(X_{k(n)1}^2)_{n\in\N}$ is uniformly integrable. Consequently, 
		\begin{align*}
			\lim\limits_{n\to\infty} \E [X_{k(n)1}^2 ] =0 ,
		\end{align*}
		which contradicts the fact that  $\E [ X_{k1}^2]=1$ for every $1 \leq k \leq p, ~ n\in\N$. 
	
	
	\end{proof}
	
	\begin{lemma} \label{lem_deriv_laplace}
	For all $s>0$, the Laplace transform $\varphi_k(s)$ and the absolute value of its derivative $\varphi_k'(s)$ are uniformly bounded away from zero, that is,
	 \begin{align*}
	 \inf\limits_{n\in\N} \min_{1 \leq k \leq p}   \varphi_k(s)  > 0
		\quad \textnormal{ and } \quad
		\inf\limits_{n\in\N} \min_{1 \leq k \leq p}  \lb  - \varphi_k'(s) \rb > 0.
	\end{align*}
	\end{lemma} 
	\begin{proof}
	We start by proving the first statement. 
	Note that 
	\begin{align*}
		\varphi_k (s) \geq & \E \left[ \exp \lb - s \lb \sum\limits_{l \in \mathcal{I}(k)} | U_{kl} \tilde{X}_{l1} | \rb^2 \rb \right]
		 \geq \E \left[ \exp \lb - s \lb \sum\limits_{l \in \mathcal{I}(k)} | \tilde{X}_{l1} | \rb^2 \rb \right] \\
		  \geq & \E \left[ \exp \lb - s \lb \sum\limits_{l =1}^q | \tilde{X}_{l1} | \rb^2 \rb \right] > 0 \textnormal{ uniformly in } k\in\N,
	\end{align*}
		where 
	\begin{align*}
		q = \sup_{n\in\N} \max_{1 \leq k \leq p}  | \mathcal{I}(k) | < \infty.
	\end{align*}

	Next, we study the derivative of the Laplace transform. Choose $\varepsilon>0$ such that \eqref{aim_eps} from the proof of Lemma \ref{lem_laplace_bound} holds true. Then, we obtain the estimate
 	\begin{align*}
 		 - \varphi_k' ( s ) 
 		& =   \E \left[ X_{k1}^2 \exp \lb - s X_{k1}^2 \rb \right]\\ 
		&  =  \E \left[ X_{k1}^2 \exp \lb - s X_{k1}^2 \rb \1 \{ X_{k1}^2 \leq \varepsilon \} \right] 
		 + \E \left[ X_{k1}^2 \exp \lb - s X_{k1}^2 \rb \1 \{ X_{k1}^2 > \varepsilon \} \right] \\
		 &\geq \varepsilon \varphi_k (s) \PR (  X_{k1}^2 > \varepsilon ) 
		> 0,
 	\end{align*}
 	uniformly over $1 \leq k \leq p, ~ n\in\N$, 
	where we used \eqref{aim_eps} and assumption \ref{a_sparse_U}.

	\end{proof}

	\begin{lemma} \label{lem_cn}
		For all $\varepsilon>0$, the quantity
		\begin{align*}
			C_n(\varepsilon) = \max\limits_{1 \leq k \leq p} n  \int\limits_\varepsilon^\infty 
			s \varphi_k''(s) \varphi_k^{n-1} (s) ds 
		\end{align*}
		converges to zero as $n$ tends to infinity. 
	\end{lemma} 
	\begin{proof}
	 By using Lemma \ref{lem_deriv_laplace} and Lemma \ref{lem_laplace_bound} and considering the estimate 
	 \begin{align*}
	 	C_n(\varepsilon) 
	 	\leq \max_{1\le k\le p} \lb \frac { n \varphi_k^{n} (\varepsilon)  }{ \varphi_k(\varepsilon) } \rb 
	 	\max_{1\le k\le p} \lb \int\limits_\varepsilon^\infty s \varphi_k''(s) ds  \rb 
	 	= o(1) \max_{1\le k\le p} \lb \int\limits_\varepsilon^\infty s \varphi_k''(s) ds  \rb  ,
	 \end{align*}
	  it is sufficient to prove that 
	 \begin{align*}
	 	\max_{1\le k\le p} \int\limits_\varepsilon^\infty s \varphi_k''(s) ds < \infty.
	 \end{align*}
	 We obtain via partial integration and dominated convergence
	 \begin{align*}
	 	\int\limits_\varepsilon^\infty s \varphi_k''(s) ds 
	 	= - \varepsilon \varphi_k'(\varepsilon) + \varphi_k(\varepsilon) 
	 	\lesssim 1,
	 \end{align*}
	 where the last inequality holds uniformly over $k\in\N$. 
	\end{proof}
	
	\begin{lemma} \label{lem_int}
	Assume that $\E | \xi|^{2+\delta} <\infty$ for some $\delta >0$. 
	For all $\varepsilon>0$, the integral
	\begin{align*}
	 	D_n ( \varepsilon) = \max\limits_{1 \leq k \leq p} \int\limits_0^\varepsilon n \E \left[   X_{k1}   \exp\lb - s  X_{k1}^2 \rb \1\{ | X_{k1}| \leq \sqrt{n} \delta_n \} \right] \lb \varphi_{k}(s) \rb^{n-1} s^{-\frac{1}{2}}  ds
	\end{align*} 
	converges to zero, as $n$ tends to infinity.
	\end{lemma}
	\begin{proof}
	Performing a Taylor expansion for $\exp(-s)$, $s>0$, we get
	\begin{align} \label{taylor}
		\exp(-s) = 1 - ( 1 - \exp( - \zeta(s) )  s,
	\end{align}
	where $\zeta(s) \in (0,s)$, which implies
	\begin{align*}
		&\left| \E \left[   X_{k1}   \exp\lb - s  X_{k1}^2 \rb \1\{ | X_{k1}| \leq \sqrt{n} \delta_n \} \right] \right| \\
		&=   \left| \E \left[ X_{k1} \left\{  1 - \lb 1 - \exp( - \zeta(sX_{k1}^2 ) ) \rb s X_{k1}^2 \right\} \1\{ | X_{k1}| \leq \sqrt{n} \delta_n \} \right] \right| \\
		&\leq  \left| \E \left[ X_{k1}  \1\{ | X_{k1}| \leq \sqrt{n} \delta_n \} \right] \right|
		+ s \E \left| X_{k1}^3  \1\{ | X_{k1}| \leq \sqrt{n} \delta_n \} \right|.
	\end{align*}
	Hence, we get
		$| D_n(\varepsilon) | 
		\leq  D_{n,1} (\varepsilon) + D_{n,2} ( \varepsilon)$,
	where
	\begin{align*}
	D_{n,1}(\varepsilon) = & \max\limits_{1 \leq k \leq p} \int\limits_0^\varepsilon n \left| \E \left[   X_{k1}    \1\{ | X_{k1}| \leq \sqrt{n} \delta_n \} \right] \right| \lb \varphi_{k}(s) \rb^{n-1} s^{-\frac{1}{2}}  ds,  \\
	D_{n,2} ( \varepsilon) = &  \max\limits_{1 \leq k \leq p} \int\limits_0^\varepsilon n s \E \left|   X_{k1}^3   \1\{ | X_{k1}| \leq \sqrt{n} \delta_n \} \right| \lb \varphi_{k}(s) \rb^{n-1} s^{-\frac{1}{2}}  ds.
	\end{align*}
	In the following, we will show that $D_{n,1}(\varepsilon) = o(1)$ and $D_{n,2} (\varepsilon) = o(1)$. 
	Using 
	\begin{align*}
		 \E \left[   X_{k1}    \1\{ | X_{k1}| \leq \sqrt{n} \delta_n \} \right] 
		= - \E \left[   X_{k1}    \1\{ | X_{k1}| > \sqrt{n} \delta_n \} \right],
	\end{align*}
	we find the following estimate using \Holder~inequality and Lemma \ref{lem_pr_delta}
	\begin{align*}
		\left| \E \left[   X_{k1}    \1\{ | X_{k1}| \leq \sqrt{n} \delta_n \} \right]  \right| 
		&\leq  \lb \E | X_{k1} |^{2+\delta} \rb^{\frac{1}{2 + \delta}}
		\lb \PR \lb | X_{k1}| > \sqrt{n} \delta_n \rb \rb^{\frac{1 + \delta}{2 + \delta }}\\
		&\lesssim \lb \PR \lb | X_{k1}| > \sqrt{n} \delta_n \rb \rb^{\frac{1 + \delta}{2 + \delta }}
		=  o \lb n^{- \frac{1+\delta}{2 + \delta} } \rb,
	\end{align*}
	which implies for sufficiently large $n$ 
	\begin{align} \label{est4}
		n^{\frac{1 + \delta}{2 + \delta }} \left| \E \left[   X_{k1}    \1\{ | X_{k1}| \leq \sqrt{n} \delta_n \} \right]  \right|  \lesssim 1.
	\end{align}
	Moreover, we have using \eqref{est_varphi}
	\begin{align}
		n^{\frac{1}{2 + \delta}} \varphi_k^{n-1} (s) s^{- \frac{1}{2}}
		= & \lb n s \varphi_k^{n-1} (s) \rb^{\frac{1}{2 + \delta}} s^{- \frac{1}{2} - \frac{1}{2 + \delta} } \lb \varphi_k^{ n-1}(s) \rb^{1 - \frac{1}{2 + \delta}} 
		\lesssim s^{-\frac{1}{2} - \frac{1}{2 + \delta}}, \label{est5}
	\end{align}
	which is integrable on $(0,\varepsilon)$. 
	Combining \eqref{est4} and \eqref{est5}, we may apply the dominated convergence theorem for the integral in $D_{n,1}(\varepsilon)$ and, by Lemma \ref{lem_laplace_bound}, we
	conclude that $D_{n,1}(\varepsilon) = o(1)$. \\
	Investigating $D_{n,2}(\varepsilon)$ further, we see that
	\begin{align*}
		\E \left|   X_{k1}^3   \1\{ | X_{k1}| \leq \sqrt{n} \delta_n \} \right| 
		&\leq  \E \left| X_{k1} \right|^{2+\delta} \lb \sqrt{n} \delta_n \rb^{1 - \delta}
		\lesssim  n^{\frac{1}{2} ( 1 - \delta) }
	\end{align*}
	and
	\begin{align}
		\lb s n \rb^{1 + \frac{1}{2} (1 - \delta) } 
		\varphi_k^{n-1}(s) 
		&\lesssim  \lb s n \rb^{1 + \frac{1}{2} (1 - \delta) } 
		\varphi_k^{n}(s)
		\leq  \lb s n \rb^{1 + \frac{1}{2} (1 - \delta) }  e^{n s \varphi_k'(s) } \nonumber \\
		&\lesssim \lb s n (- \varphi_k'(s)) \rb^{1 + \frac{1}{2} (1 - \delta) }  e^{n s \varphi_k'(s) } 
	 \lesssim   1  , \label{ineq_nsphi}
	\end{align}
	where we used Lemma 3.3 in \cite{fuchs:joffe:teugels:2001} and Lemma \ref{lem_deriv_laplace}. This implies for the integrand in $D_{n,2}(\varepsilon)$
	\begin{align*}
		n s \E \left|   X_{k1}^3   \1\{ | X_{k1}| \leq \sqrt{n} \delta_n \} \right| \lb \varphi_{k}(s) \rb^{n-1} s^{-\frac{1}{2}} 
		&\lesssim  s^{-\frac{1}{2} -  \frac{1}{2} ( 1 - \delta) } 
		= s^{- 1 + \frac{1}{2}\delta} ,
	\end{align*}
	which is integrable on $(0,\varepsilon)$.
	Thus, by an application of the dominated convergence theorem and Lemma \ref{lem_laplace_bound}, it follows that $D_{n,2}(\varepsilon) =o(1)$. 
	\end{proof}

	\bibliography{libraryFeb2020}

\begin{thebibliography}{10}

\bibitem{anderson2003}
{\sc Anderson, T.~W.}
\newblock {\em An introduction to multivariate statistical analysis},
  second~ed.
\newblock Wiley Series in Probability and Mathematical Statistics: Probability
  and Mathematical Statistics. John Wiley \& Sons, Inc., New York, 1984.

\bibitem{auffinger:arous:peche:2009}
{\sc Auffinger, A., Ben~Arous, G., and P{\'e}ch{\'e}, S.}
\newblock Poisson convergence for the largest eigenvalues of heavy tailed
  random matrices.
\newblock {\em Ann. Inst. Henri Poincar\'e Probab. Stat. 45}, 3 (2009),
  589--610.

\bibitem{bai:silverstein:2010}
{\sc Bai, Z., and Silverstein, J.~W.}
\newblock {\em Spectral Analysis of Large Dimensional Random Matrices},
  second~ed.
\newblock Springer Series in Statistics. Springer, New York, 2010.

\bibitem{bai:zhou:2008}
{\sc Bai, Z., and Zhou, W.}
\newblock Large sample covariance matrices without independence structures in
  columns.
\newblock {\em Statist. Sinica 18}, 2 (2008), 425--442.

\bibitem{BS04}
{\sc Bai, Z.~D., and Silverstein, J.}
\newblock {CLT} for linear spectral statistics of large-dimensional sample
  covariance matrices.
\newblock {\em Ann. Probab. 32\/} (2004), 553--605.

\bibitem{BaiYin88a}
{\sc Bai, Z.~D., and Yin, Y.~Q.}
\newblock Necessary and sufficient conditions for almost sure convergence of
  the largest eigenvalue of a {Wigner} matrix.
\newblock {\em Ann. Probab. 16\/} (1988), 1729--1741.

\bibitem{bai:yin:1993}
{\sc Bai, Z.~D., and Yin, Y.~Q.}
\newblock Limit of the smallest eigenvalue of a large-dimensional sample
  covariance matrix.
\newblock {\em Ann. Probab. 21}, 3 (1993), 1275--1294.

\bibitem{basrak:heiny:jung:2020}
{\sc Basrak, B., Cho, Y., Heiny, J., and Jung, P.}
\newblock Extreme eigenvalue statistics of $ m $-dependent heavy-tailed
  matrices.
\newblock {\em arXiv preprint arXiv:1910.08511\/} (2019).

\bibitem{bingham:goldie:teugels:1987}
{\sc Bingham, N.~H., Goldie, C.~M., and Teugels, J.~L.}
\newblock {\em Regular Variation}, vol.~27 of {\em Encyclopedia of Mathematics
  and its Applications}.
\newblock Cambridge University Press, Cambridge, 1987.

\bibitem{blum1961}
{\sc Blum, J.~R., Kiefer, J., and Rosenblatt, M.}
\newblock {Distribution Free Tests of Independence Based on the Sample
  Distribution Function}.
\newblock {\em Ann. Math. Statist. 32}, 2 (1961), 485 -- 498.

\bibitem{brockwell:davis:1991}
{\sc Brockwell, P.~J., and Davis, R.~A.}
\newblock {\em Time series: theory and methods}, second~ed.
\newblock Springer Series in Statistics. Springer-Verlag, New York, 1991.

\bibitem{davis:heiny:mikosch:xie:2016}
{\sc Davis, R.~A., Heiny, J., Mikosch, T., and Xie, X.}
\newblock Extreme value analysis for the sample autocovariance matrices of
  heavy-tailed multivariate time series.
\newblock {\em Extremes 19}, 3 (2016), 517--547.

\bibitem{donoho:2000}
{\sc Donoho, D.}
\newblock High-dimensional data analysis: the curses and blessings of
  dimensionality.
\newblock {\em Technical Report, Stanford University\/} (2000).

\bibitem{elkaroui:2009}
{\sc El~Karoui, N.}
\newblock Concentration of measure and spectra of random matrices: applications
  to correlation matrices, elliptical distributions and beyond.
\newblock {\em Ann. Appl. Probab. 19}, 6 (2009), 2362--2405.

\bibitem{Fan2006}
{\sc Fan, J., and Li, R.}
\newblock Statistical challenges with high dimensionality: Feature selection in
  knowledge discovery.
\newblock {\em Proceedings of the International Congress of Mathematicians,
  Madrid 3\/} (2006).

\bibitem{fuchs:joffe:teugels:2001}
{\sc Fuchs, A., Joffe, A., and Teugels, J.}
\newblock Expectation of the ratio of the sum of squares to the square of the
  sum: exact and asymptotic results.
\newblock {\em Teor. Veroyatnost. i Primenen. 46}, 2 (2001), 297--310.

\bibitem{Gao2017}
{\sc Gao, J., Han, X., Pan, G., and Yang, Y.}
\newblock High dimensional correlation matrices: the central limit theorem and
  its applications.
\newblock {\em Journal of the Royal Statistical Society. Series B: Statistical
  Methodology 79}, 3 (2017), 677--693.

\bibitem{gine:goetze:mason:1997}
{\sc Gin{\'e}, E., G{\"o}tze, F., and Mason, D.~M.}
\newblock When is the {S}tudent {$t$}-statistic asymptotically standard normal?
\newblock {\em Ann. Probab. 25}, 3 (1997), 1514--1531.

\bibitem{heiny:2022}
{\sc Heiny, J.}
\newblock Large sample correlation matrices: a comparison theorem and its
  applications.
\newblock {\em Electron. J. Probab. 27\/} (2022), 1--20.

\bibitem{heiny:mikosch:2017:iid}
{\sc Heiny, J., and Mikosch, T.}
\newblock Eigenvalues and eigenvectors of heavy-tailed sample covariance
  matrices with general growth rates: {T}he iid case.
\newblock {\em Stochastic Process. Appl. 127}, 7 (2017), 2179--2207.

\bibitem{heiny:mikosch:2017:corr}
{\sc Heiny, J., and Mikosch, T.}
\newblock Almost sure convergence of the largest and smallest eigenvalues of
  high-dimensional sample correlation matrices.
\newblock {\em Stochastic Processes and their Applications 128}, 8 (2018),
  2779--2815.

\bibitem{heiny:parolya:2022}
{\sc Heiny, J., and Parolya, N.}
\newblock Log determinant of large correlation matrices under infinite fourth
  moment.
\newblock {\em arXiv preprint arXiv:2112.15388\/} (2021).

\bibitem{heiny:yao:2020}
{\sc Heiny, J., and Yao, J.}
\newblock Limiting distributions for eigenvalues of sample correlation matrices
  from heavy-tailed populations.
\newblock {\em arXiv preprint arXiv:2003.03857\/} (2020).

\bibitem{Hoeffding1948}
{\sc Hoeffding, W.}
\newblock {A Non-Parametric Test of Independence}.
\newblock {\em The Annals of Mathematical Statistics 19}, 4 (1948), 546 -- 557.

\bibitem{jiang:2004}
{\sc Jiang, T.}
\newblock The limiting distributions of eigenvalues of sample correlation
  matrices.
\newblock {\em Sankhy\=a 66}, 1 (2004), 35--48.

\bibitem{johnstone:2001}
{\sc Johnstone, I.~M.}
\newblock On the distribution of the largest eigenvalue in principal components
  analysis.
\newblock {\em Ann. Statist. 29}, 2 (2001), 295--327.

\bibitem{Johnstone2006}
{\sc Johnstone, I.~M.}
\newblock High dimensional statistical inference and random matrices.
\newblock {\em Proceedings of the International Congress of Mathematicians,
  Madrid\/} (2006).

\bibitem{Kendall1938}
{\sc Kendall, M.~G.}
\newblock A new measure of rank correlation.
\newblock {\em Biometrika 30}, 1/2 (1938), 81--93.

\bibitem{kevei:2021}
{\sc Kevei, P.}
\newblock On a conjecture of {S}eneta on regular variation of truncated
  moments.
\newblock {\em Publ. Inst. Math. (Beograd) (N.S.) 109(123)\/} (2021), 77--82.

\bibitem{lai:rayner:hutchinson:1999}
{\sc Lai, C.-D., Rayner, J.~C., and Hutchinson, T.}
\newblock Robustness of the sample correlation-the bivariate lognormal case.
\newblock {\em Advances in Decision Sciences 3}, 1 (1999), 7--19.

\bibitem{marchenko:pastur:1967}
{\sc Mar{\v{c}}enko, V.~A., and Pastur, L.~A.}
\newblock Distribution of eigenvalues in certain sets of random matrices.
\newblock {\em Mat. Sb. (N.S.) 72 (114)\/} (1967), 507--536.

\bibitem{mason:zinn:2005}
{\sc Mason, D.~M., and Zinn, J.}
\newblock {W}hen does a randomly weighted self-normalized sum converge in
  distribution?
\newblock {\em Electron. Comm. Probab. 10\/} (2005), 297 (electronic).

\bibitem{muirhead1982}
{\sc Muirhead, R.~J.}
\newblock {\em Aspects of Multivariate Statistical Theory}.
\newblock John Wiley \& Sons, New York, 1982.

\bibitem{heiny:parolya:kurowicka:2021}
{\sc Parolya, N., Heiny, J., and Kurowicka, D.}
\newblock Logarithmic law of large random correlation matrix.
\newblock {\em arXiv preprint arXiv:2103.13900\/} (2021).

\bibitem{Peche2012}
{\sc Péché, S.}
\newblock Universality in the bulk of the spectrum for complex sample
  covariance matrices.
\newblock {\em Annales de l'institut Henri Poincare (B) Probability and
  Statistics 48}, 1 (2012), 80--106.

\bibitem{Pearson1920}
{\sc Pearson, K.}
\newblock Notes on the history of correlation.
\newblock {\em Biometrika 13}, 1 (1920), 25--45.

\bibitem{PillaiYin2014}
{\sc Pillai, N., and Yin, J.}
\newblock Universality of covariance matrices.
\newblock {\em Annals of Applied Probability 24}, 3 (2014), 935--1001.

\bibitem{priestley:1981}
{\sc Priestley, M.~B.}
\newblock {\em Spectral analysis and time series. {V}ols. 1 and 2}.
\newblock Academic Press, Inc. [Harcourt Brace Jovanovich, Publishers],
  London-New York, 1981.
\newblock Univariate series, Probability and Mathematical Statistics.

\bibitem{silverstein:bai:1995}
{\sc Silverstein, J.~W., and Bai, Z.~D.}
\newblock On the empirical distribution of eigenvalues of a class of
  large-dimensional random matrices.
\newblock {\em J. Multivariate Anal. 54}, 2 (1995), 175--192.

\bibitem{tikhomirov:2015}
{\sc Tikhomirov, K.}
\newblock The limit of the smallest singular value of random matrices with
  i.i.d. entries.
\newblock {\em Adv. Math. 284\/} (2015), 1--20.

\bibitem{yao:zheng:bai:2015}
{\sc Yao, J., Zheng, S., and Bai, Z.}
\newblock {\em Large sample covariance matrices and high-dimensional data
  analysis}.
\newblock Cambridge Series in Statistical and Probabilistic Mathematics.
  Cambridge University Press, New York, 2015.

\bibitem{yaskov:2016}
{\sc Yaskov, P.}
\newblock Necessary and sufficient conditions for the {M}archenko-{P}astur
  theorem.
\newblock {\em Electron. Commun. Probab. 21\/} (2016), Paper No. 73, 8.

\end{thebibliography}
\end{document}